\documentclass[11pt]{article}
\usepackage{latexsym,amsfonts,amssymb,amsmath,amsthm}

\usepackage{color,comment}
\usepackage{authblk}

\begin{document}

\newcommand{ \bl}{\color{blue}}
\newcommand {\rd}{\color{red}}
\newcommand{ \bk}{\color{black}}
\newcommand{ \gr}{\color{OliveGreen}}
\newcommand{ \mg}{\color{RedViolet}}

\newcommand{\norm}[1]{||#1||}
\newcommand{\normo}[1]{|#1|}
\newcommand{\secn}[1]{\addtocounter{section}{1}\par\medskip\noindent
 {\large \bf \thesection. #1}\par\medskip\setcounter{equation}{0}}

\newcommand{\secs}[1]{\addtocounter {section}{1}\par\medskip\noindent
 {\large \bf  #1}\par\medskip\setcounter{equation}{0}}
\renewcommand{\theequation}{\thesection.\arabic{equation}}

\setlength{\baselineskip}{16pt}

\newtheorem{theorem}{Theorem}[section]
\newtheorem{lemma}{Lemma}[section]
\newtheorem{proposition}{Proposition}[section]
\newtheorem{definition}{Definition}[section]
\newtheorem{example}{Example}[section]
\newtheorem{corollary}{Corollary}[section]

\newtheorem{remark}{Remark}[section]

\numberwithin{equation}{section}

\def\p{\partial}
\def\I{\textit}
\def\R{\mathbb R}
\def\C{\mathbb C}
\def\u{\underline}
\def\l{\lambda}
\def\a{\alpha}
\def\O{\Omega}
\def\e{\epsilon}
\def\ls{\lambda^*}
\def\D{\displaystyle}
\def\wyx{ \frac{w(y,t)}{w(x,t)}}
\def\imp{\Rightarrow}
\def\tE{\tilde E}
\def\tX{\tilde X}
\def\tH{\tilde H}
\def\tu{\tilde u}
\def\d{\mathcal D}
\def\aa{\mathcal A}
\def\DH{\mathcal D(\tH)}
\def\bE{\bar E}
\def\bH{\bar H}
\def\M{\mathcal M}
\renewcommand{\labelenumi}{(\arabic{enumi})}

\def\disp{\displaystyle}
\def\undertex#1{$\underline{\hbox{#1}}$}
\def\card{\mathop{\hbox{card}}}
\def\sgn{\mathop{\hbox{sgn}}}
\def\exp{\mathop{\hbox{exp}}}
\def\OFP{(\Omega,{\cal F},\PP)}
\newcommand\JM{Mierczy\'nski}
\newcommand\RR{\ensuremath{\mathbb{R}}}
\newcommand\CC{\ensuremath{\mathbb{C}}}
\newcommand\QQ{\ensuremath{\mathbb{Q}}}
\newcommand\ZZ{\ensuremath{\mathbb{Z}}}
\newcommand\NN{\ensuremath{\mathbb{N}}}
\newcommand\PP{\ensuremath{\mathbb{P}}}
\newcommand\abs[1]{\ensuremath{\lvert#1\rvert}}

\newcommand\normf[1]{\ensuremath{\lVert#1\rVert_{f}}}
\newcommand\normfRb[1]{\ensuremath{\lVert#1\rVert_{f,R_b}}}
\newcommand\normfRbone[1]{\ensuremath{\lVert#1\rVert_{f, R_{b_1}}}}
\newcommand\normfRbtwo[1]{\ensuremath{\lVert#1\rVert_{f,R_{b_2}}}}
\newcommand\normtwo[1]{\ensuremath{\lVert#1\rVert_{2}}}
\newcommand\norminfty[1]{\ensuremath{\lVert#1\rVert_{\infty}}}
\newcommand{\ds}{\displaystyle}

\title{Pointwise persistence in  full chemotaxis models with logistic source on bounded heterogeneous environments}

\author[1]{Tahir Bachar Issa}
\author[2]{Wenxian Shen\thanks{Partially supported by the NSF grant DMS--1645673}}
\affil[1]{Data Institute, University of San Francisco, CA 94117, USA.}
\affil[2]{Department of Mathematics and Statistics, Auburn University, AL 36849, USA.}

\date{}
\maketitle

\noindent {\bf Abstract.} {The current paper is concerned with pointwise persistence  in   full chemotaxis models with local as well as nonlocal time and space dependent logistic source in  bounded domains. We first prove the global existence and boundedness of  nonnegative classical solutions under some conditions on the coefficients in the models.
Next, under the same conditions on the coefficients, we show that pointwise { persistence} occurs, that is, any globally defined positive solution is bounded below by a positive constant independent of its initial condition when the time is large enough.
It should be pointed out that in \cite{TaoWin15}, the authors established the persistence of mass for globally defined positive solutions, which
 indicates that any extinction phenomenon, if occurring at all, necessarily
must be
spatially local { in }nature, whereas the population as a whole always persists. The pointwise persistence proved in the current paper implies that
not only the population as a whole persists,  but also it persists at any location
eventually. It also implies the existence of strictly positive entire solutions.
}
\medskip

\noindent {\bf Key words.} {Full chemotaxis model, global existence,  pointwise persistence,  positive entire solutions, comparison principle.}

\medskip

\noindent {\bf 2010 Mathematics Subject Classification.} 35A01, 35A02, 35B40, {35K57, 35Q92, 92C17}

\section{Introduction and the statements of the main results}
\label{S:intro}
In this paper, we study the dynamics of the following full chemotaxis model,
\begin{equation}
 \label{u-v-eq00}
\begin{cases}
u_t=\Delta u-\chi\nabla\cdot (u \nabla v)+u\Big(a_0(t,x)-a_1(t,x)u-a_2(t,x)\int_{\Omega}u\Big),\quad x\in \Omega\cr
\tau v_t=\Delta v-\lambda v +\mu u,\quad x\in \Omega \cr
\frac{\p u}{\p n}=\frac{\p v}{\p n}=0,\quad x\in\p\Omega,
\end{cases}
 \end{equation}
where $\Omega \subset \mathbb{R}^n(n\geq 1)$ is a bounded domain with smooth boundary; $u(x,t)$ represents the population density
of a  mobile species, and $v(x,t)$ is the population density of the chemical substance created by the mobile species;   $\chi \in \mathbb{R}$ represents
the chemotactic sensitivity  effect on the mobile species;  $\tau$ is a positive constant related to the diffusion rate of the chemical substance; and  $\lambda$  represents the degradation rate of the {  chemical substance} and $\mu$ is the rate at which the mobile species produces the {chemical substance}.
 The term $ u\big(a_0(t,x)-a_1(t,x)u-a_2(t,x)\int_{\Omega}u\big)$  in the first equation of \eqref{u-v-eq00}
 is referred to as the logistic source or
 logistic growth describing  the competition of the individuals of
the species for the resources of the environment and the cooperation
to survive. The  coefficient $a_0$ induces an exponential growth for
low density population and the term $a_1u$ describes a local competition of the species. { When the population density is large}, the
competitive effect of the local term $a_1u$  becomes more influential. The non-local term $a_2\int_{\Omega}u$ describes the influence of the total mass of the species in the growth of the population. If $a_2>0,$ we have a competitive term which limits such growth and when $a_2<0$ the individuals cooperate globally to survive. In the last case, the individuals compete locally but cooperate globally and the effects of $a_1u$ and $a_2\int_{\Omega}u$ balance the system.

Chemotaxis, the oriented movements of mobile species toward the increasing or decreasing concentration of a signaling chemical substance,  has a crucial role in a wide range of biological phenomena such as immune system response, embryo development, tumor growth,  population dynamics, gravitational collapse, etc. (see  \cite{ISM04, DAL1991}).
 At the beginning of 1970s, Keller and Segel proposed a celebrated mathematical model, referred to as the classical Keller-Segel model, to describe the aggregation process of Dictyostelium discoideum, a soil-living amoebea \cite{KS1970, KS71}.  System \eqref{u-v-eq00} with $a_i \equiv 0$ ($i=0,1,2$) reduces to the classical Keller-Segel model.

 {  Consider \eqref{u-v-eq00}.  Central problems include
 global existence of classical/weak solutions with given initial  data;  finite-time blow-up; asymptotic behavior of globally defined solutions such as persistence and convergence as time goes to infinity;  etc.
   A large amount of research has been carried out toward many of these central problems in various chemotaxis models (see  \cite{NBYTMW05, THKJP09, H03, Win2013} for some  survey on the study of various chemotaxis models)}.
  For example, it is well known that finite-time blow-up of some classical solution may occur in the classical Keller-Segel model and its variants in space dimension $n\geq 2$  (see \cite{CoEsVi11, HVJ1997a, JaW92, W2011b} for one species chemotaxis model and \cite{ASV2009} for two species chemotaxis models). It is   known that logistic sources of Lotka-Volterra type may preclude such blow-up phenomenon (see \cite{ITBWS16, RBSWS17a, TW07} for one species and \cite{ITBRS17, ITBWS17a, ITBWS17b, NT13, TW12} for two species), and that, at least numerically, chemotaxis with logistic sources may also exhibit quite a rich variety of colorful dynamical features, up to periodic and even chaotic dynamics (see  \cite{kuto_PHYSD, PaHi}).

  The objective of the current paper is to investigate the {\it pointwise persistence} in \eqref{u-v-eq00}, which is motivated by the works \cite{PaHi} and \cite{TaoWin15}. In \cite{PaHi}, spatio-temporal chaotic dynamics in some special case of \eqref{u-v-eq00} is studied numerically.
    { A  phenomenon suggested by the numerical simulations in \cite{PaHi} consists in the ability of (1.1) to enforce
asymptotic smallness of the cell population density, undistinguishable from extinction, in large
spatial regions (see e.g. Fig. 7(d) in \cite{PaHi}). As commented in \cite{TaoWin15}, in the case that $a_0(t,x)\equiv a_0>0$, $a_1(t,x)\equiv a_1>0$, and $a_2(t,x)\equiv 0$,  such types of solution behavior, seemingly paradoxical
due to the presence of the reproduction term $a_0u$ dominating e.g. the death term  $-a_1u^2$  at
small
densities, clearly reflect a truly cross-diffusive effect  in  view of the evident fact that when
$\chi=0$,
all positive solutions of the resulting decoupled problem approach the spatially homogeneous nontrivial
state $(\frac{a_0}{a_1},\frac{\mu}{\lambda}\frac{a_0}{a_1})$. In \cite{TaoWin15}, the authors
proved that any such extinction phenomenon must be
localized in space, and  that the population as a whole always persists, which is called {\it  persistence of mass} in \cite{TaoWin15}.
Both mathematically and biologically, it is interesting to know
 whether the population actually persists pointwise. In this paper, we will give a confirmed answer for parameters
in certain region, which implies that the cell population may become very small at some time and some location, but it persists at any location
eventually. }

{ To state our main results on the pointwise persistence in \eqref{u-v-eq00}, we first present the following lemma on   the maximal Sobolev regularity.

\begin{lemma}\cite[Lemma 2.2]{JZYYLGBXZ2018}
\label{prelimiaries-lm-00}
Suppose $\gamma \in (1,+\infty)$ and $g \in L^\gamma((0,T);L^\gamma(\Omega)).$  Assume that $v$ is a solution of the following initial boundary value problem,
\begin{equation}
\label{prelimiaries-eq-00}
 \begin{cases}
\tau v_t-\Delta v+\lambda v=g,\cr
\frac{\p v}{\p n}=0,\cr
v(x,0)=v_0(x).
\end{cases}
 \end{equation}
 Then there exists a positive constant $C_{\gamma}$ such that if $s_0 \in [0,T),$ $v(\cdot,s_0) \in W^{2,\gamma}(\Omega)$ with $\frac{\p v(\cdot,s_0)}{\p n}=0,$ then
 \begin{align}
\label{prelimiaries-eq-01}
   &\int_{s_0}^Te^{\gamma s}\|\Delta v(\cdot,s)\|^\gamma_{L^\gamma(\Omega)} ds \nonumber\\
  &\leq C_{\gamma} \Big(\int_{s_0}^Te^{\gamma s}\|g(\cdot,s)\|^\gamma_{L^\gamma(\Omega)} ds+e^{\gamma s_0}\big(\|v(\cdot,s_0)\|^\gamma_{L^\gamma(\Omega)}+\|\Delta v(\cdot,s_0)\|^\gamma_{L^\gamma(\Omega)}\big)\Big).
\end{align}
\end{lemma}

The constant $C_\gamma$ such that \eqref{prelimiaries-eq-01} holds is not unique. In the following, we always assume that
$C_\gamma$ is the smallest positive constant such that \eqref{prelimiaries-eq-01} holds.
}

Next, we  introduce some notations and definitions.  Throughout the paper, we  put
\vspace{-0.05in} \begin{equation}
 \label{a-i-sup-inf-eq1}
a_{i,\inf}=\inf _{ t \in\RR,x \in\bar{\Omega}}a_i(t,x),\quad a_{i,\sup}=\sup _{ t \in\RR,x \in\bar{\Omega}}a_i(t,x),
 \vspace{-0.05in}\end{equation}
 \begin{equation}
 \label{a-i-sup-inf-eq2}
a_{i,\inf}(t)=\inf _{x \in\bar{\Omega}}a_i(t,x),\quad a_{i,\sup}(t)=\sup _{x \in\bar{\Omega}}a_i(t,x),
\vspace{-0.05in} \end{equation}
 unless specified otherwise.

For given $t_0\in\RR$, $u_0\in C(\bar\Omega)$,  and $v_0\in W^{1,\infty}(\Omega)$ with $u_0> 0$ and  $v_0\geq 0$, we denote a classical solution
$(u(t,x),v(t,x))$ of \eqref{u-v-eq00}  by $(u(t,x;t_0,u_0,v_0),v(t,x;t_0,u_0,v_0))$ if it is defined on $[t_0,t_0+T)$ for some $T>0$
and satisfies
\begin{equation}
\label{initial-cond-eq}
\lim_{t\to t_0^+} (u(t,\cdot;t_0,u_0,v_0),v(t,\cdot;t_0,u_0,v_0))=(u_0(\cdot),v_0(\cdot))
\end{equation}
{ in $C(\bar\Omega)\times W^{1,q}(\Omega)$ for any $q>n.$ } In such case, $(u(t,x;t_0,u_0,v_0),v(t,x;t_0,u_0,v_0))$ is called
{\it the solution
of \eqref{u-v-eq00} on $[t_0,t_0+T)$ with initial condition} $(u(t_0,x),v(t_0,x))=(u_0(x),v_0(x))$. If $T$ can be chosen to be $\infty$, we say
the solution of \eqref{u-v-eq00} with initial condition $(u(t_0,x),v(t_0,x))=(u_0(x),v_0(x))$ exists globally.
A solution $(u(x,t),v(x,t))$ of \eqref{u-v-eq00} defined for all $t\in\RR$ is called an {\it entire solution}.

\begin{definition}
\label{coex-persist-def}
We say that {\rm pointwise persistence} occurs in \eqref{u-v-eq00} if there is $\eta>0$ such that for any $t_0\in\RR$, $u_0\in C(\bar\Omega),$ and $v_0\in W^{1,\infty}(\Omega)$ with $u_0> 0$ and  $v_0\geq 0$, $(u(t,x;t_0,u_0,v_0)$, $v(t,x;t_0,u_0,v_0))$ exists globally, and
there is {$\tau(u_0,v_0)>0$}  such that
\begin{equation}\label{attracting-set-for-u-w-eq00}
 u(x,t;t_0,u_0,v_0) \ge  \eta,\quad \forall\,\, {x\in \Omega},\,\, t\ge t_0+\tau(u_0,v_0).
\end{equation}
\end{definition}

For convenience, we introduce the following two standing hypotheses.

\medskip

\noindent {\bf (H1)}  $a_{1,\inf}>\inf_{q>\max\{1,\frac{n}{2}\}}\Big(\frac{q-1}{q}(C_{q+1})^{\frac{1}{q+1}} \mu^{\frac{1}{q+1}}\Big)|\chi|$
 and  $ \inf_{t \in \mathbb{R}}\big(a_{1,\inf}(t)-{ |\Omega|(a_{2,\inf}(t))_-}\big)>0$.

\medskip

\noindent {\bf (H2)}  { $\Omega$ is convex, $\tau=1$,  $a_{1,\inf}>\frac{n\mu { |\chi|}}{4}$},  and $ \inf_{t \in \mathbb{R}}\big(a_{1,\inf}(t)-{ |\Omega|(a_{2,\inf}(t))_-}\big)>0.$

\medskip

We now state our main results.  The first theorem is on the global existence and boundedness of nonnegative classical solutions of system \eqref{u-v-eq00}, which is fundamental for the study of pointwise persistence.

\begin{theorem}{ (Global Existence)}
\label{thm-global-000}
  Assume that {\bf (H1)} or {\bf (H2)} holds.  Then for any $t_0 \in \R$  and $(u_0,v_0) \in C(\bar\Omega )\times W^{1,\infty}(\Omega)$ { with $u_0,v_0 \geq 0$}, \eqref{u-v-eq00} has a unique bounded {globally defined} classical solution $(u(\cdot,t;t_0,u_0,v_0)$, $v(\cdot,t;t_0,u_0,v_0))$.  Furthermore,  there are positive numbers $M_1$ and $M_2$ independent of $t_0,u_0,v_0$ and there are $t^2(u_0,v_0)>t^1(u_0,v_0)>0$ such that
  \begin{equation}
  \label{m1-eq}
  \int_\Omega u(x,t;t_0,u_0,v_0)\le M_1\quad \forall\, t\ge t_0+t^1(u_0,v_0)
  \end{equation}
  and
  \begin{equation}
  \label{m2-eq}
  \|u(\cdot,t;t_0u_0,v_0)\|_{\infty}\le M_2\quad \forall\, t\ge t_0+t^2(u_0,v_0).
  \end{equation}
\end{theorem}

The second theorem is on pointwise persistence.

\begin{theorem}{ (Pointwise persistence)}
\label{thm-persistence-entire solution-000}
Suppose that {\bf (H1)} or {\bf (H2)} holds. Then pointwise persistence occurs in \eqref{u-v-eq00}.
\end{theorem}

Applying the above pointwise persistence theorem, we obtain the third theorem on the existence of strictly positive entire solutions
of \eqref{u-v-eq00}.

\begin{theorem} (Strictly positive entire solutions)
\label{thm-entire-solution-00}
Assume  that {\bf (H1)} or {\bf (H2)} holds. System \eqref{u-v-eq00} has a positive entire solution  $(u^*(x,t),v^*(x,t))$
satisfying
\begin{equation}\label{attracting-set-for-u-w-eq01}
\inf_{x\in\bar\Omega,t\in\RR} u^*(x,t)>0.
\end{equation}
\end{theorem}

We conclude the introduction with the following three remarks.

First, it should be pointed out that the global existence of nonnegative classical solutions has been studied in \cite{Win2010b} and \cite{JZYYLGBXZ2018} for some special cases of system  \eqref{u-v-eq00}. To be more precise,  in \cite[Theorem 0.1]{Win2010b}, Winkler considered system  \eqref{u-v-eq00} in convex domains $\Omega$ of $\mathbb{R}^n,$ with $\tau>0,$ {  $a_0(t,x)\equiv a_0$, $a_1(t,x)\equiv a_1$, $a_2(t,x)=0$, and $\lambda=\mu=1$},  and established the global existence and boundedness of nonnegative classical solutions of system  \eqref{u-v-eq00} provided that $a_1$ is large enough.
In \cite{JZYYLGBXZ2018}, Zheng, Li, Bao and Zou extended Winkler's global existence result to  bounded domains (not necessarily convex) of $\mathbb{R}^n$ for  $\chi>0$ and show that $a_1 >\frac{(n-2)_+}{n}\chi [C_{\frac{n}{2}+1}]^{\frac{1}{\frac{n}{2}+1}}$ implies global existence of nonnegative solutions in system  \eqref{u-v-eq00}.  Theorem \ref{thm-global-000} stated in the above extends the global existence results in both \cite{Win2010b} and \cite{JZYYLGBXZ2018} to the general full chemotaxis model \eqref{u-v-eq00} with local as well as nonlocal time and space dependent logistic source.  Theorem \ref{thm-global-000} under the assumption {\bf (H1)}  can   be proved by properly  modifying arguments of \cite[Theorem 2.2]{JZYYLGBXZ2018}, and   Theorem \ref{thm-global-000} under the assumption {\bf (H2)}  can be proved by properly modifying the arguments in \cite[Lemma 3.1]{Win2014}. For the completeness, we will provide a proof of Theorem \ref{thm-global-000}.

Second, as it is mentioned in the above,
  Tao and Winkler showed in \cite{TaoWin15} that  the population as a whole  always persists for some special case of \eqref{u-v-eq00}.
  Theorem \ref{thm-persistence-entire solution-000} stated in the above
  shows that under the assumption {\bf (H1)} or {\bf (H2)}, every classical solution of \eqref{u-v-eq00} persists pointwise, which
  implies the population persistence as a whole and rules out the extinction phenomenon observed numerically.
  The pointwise persistence result obtained in Theorem \ref{thm-persistence-entire solution-000} is new.
 Theorem \ref{thm-persistence-entire solution-000} is proved by very nontrivial estimates of $\|v(\cdot,t;t_0,u_0,v_0)\|_{W^{2,\infty}(\Omega)}$
 in terms of $\|u(\cdot,t;t_0,u_0,v_0)\|_\infty$ for $t\gg t_0$ and by nontrivial application of the comparison principle for parabolic equations.

 Third, Theorem  \ref{thm-persistence-entire solution-000} implies that under the assumption
 {\bf (H1)} or {\bf (H2)} any globally defined positive solution of \eqref{u-v-eq00}
 is bounded away from zero eventually. To further study the asymptotic behavior of globally defined positive solutions, it is important to study the existence of various special positive solutions such as strictly positive entire solutions.
 In the case that $a_0(t,x)\equiv a_0$, $a_1(t,x)\equiv a_1$, and $a_2(t,x)\equiv a_2>0$, it is clear that  $(u^*(x,t),v^*(x,t))=(\frac{a_0}{a_1+a_2|\Omega|}, \frac{\mu}{\lambda}\frac{a_0}{a_1+a_2|\Omega|})$ is a strictly positive entire solution of
 \eqref{u-v-eq00}.  In \cite{Win2014}, Winkler proved the global stability of  this positive entire solution  when $\tau=\lambda=\mu=1$, $a_2(t,x)\equiv 0$, and $\Omega$ is convex. It should be pointed out that it is  a challenging problem to prove existence and stability of strictly positive entire solutions. We prove Theorem \ref{thm-entire-solution-00} by applying Theorem  \ref{thm-persistence-entire solution-000}
 together with some pullback technique.
  We leave the following as open questions: {1) If the coefficients of \eqref{u-v-eq00} are periodic in $t$ with period $T$,
   does \eqref{u-v-eq00} have  positive  periodic solutions with period $T$? 2)  When  does \eqref{u-v-eq00} have a unique  stable strictly positive entire solution?}

The rest of the paper is organized as follows. In section 2, we recall some important results to be used to prove the main  results in the paper. The global existence results are established  in section 3. In section 4, we prove our main result on pointwise persistence.
 Finally, in section 5, we show
 existence of strictly positive entire solutions of system \eqref{u-v-eq00}.

\medskip

\noindent {\bf Acknowledgment.} The authors would like to thank the referees for valuable comments and suggestions which improved the presentation of this paper considerably.

\section{Preliminary}

In this section, we present some  preliminary lemmas to be used in later sections.

Let $\Omega\subset\RR^n$ be a bounded smooth domain. For given $1\le p<\infty$, it is well known that the operator
$\Delta: D(\Delta)=\{u\in W^{2,p}(\Omega)\,|\, \frac{\p u}{\p n}|_{\p \Omega}=0\}$ generates an analytic semigroup, denoted by
$e^{t\Delta}$, on $L^p(\Omega)$. { We remark that for any given $u_0\in L^p(\Omega)$, $u(t,x;u_0):=\big(e^{t\Delta}u_0\big)(x)$   is the unique solution of
$$
\begin{cases}
u_t =\Delta u,\quad t>0,\,\, x\in\Omega\cr
\frac{\p u}{\p n}=0,\quad t>0,\,\, x\in\p \Omega
\end{cases}
$$
with $\lim_{t\to 0+} u(t,x;u_0)=u_0(x)$ in $L^p(\Omega)$.
}
\begin{lemma}
\label{estimates-lm-00}
\begin{itemize}

\item[(i)] If $1\le q\le p\le \infty$, there is $K_1(q,p)>0$ such that
\begin{equation}
\label{eq-estimate-000}
\|e^{t\Delta} v\|_p\le K_1(q,p) (1+t^{-\frac{n}{2}(\frac{1}{q}-\frac{1}{p})})\|v\|_q\quad \forall\, t>0,\,\, v\in L^q(\Omega).
\end{equation}

\item[(ii)] Fix $\alpha \in(0,1)$. For given $1<p<\infty,$  let $A=-\Delta+\alpha I$  with $D(A)=\{u \in  W^{2,p}(\Omega) \,|\,  \frac{\p u}{\p n}|_{\p\Omega}=0 \}.$ Then $A$ is sectorial in $L^p(\Omega)$ and possesses closed fractional powers $A^k$ for any $k>0,$ and
\begin{equation}
\label{eq-estimate-00}
D(A^k) \hookrightarrow W^{2,\infty} \,\, \text{if} \,\, 2k-\frac{n}{p}>2.
\end{equation}
Moreover, if $(e^{-tA})_{t\geq 0}$ denotes the corresponding analytic semigroup in $L^p(\Omega)$, then  for each $k>0$,   there is $K_2(p,k)>0$ such that
\begin{equation}
\label{eq-estimate-01}
\|A^k e^{-sA}\phi\|_{ L^p(\Omega)}\leq   K_2(p,k)s^{-k}\|\phi\|_{ L^p(\Omega)}
\end{equation}
for all $s>0$ and $\phi \in L^p(\Omega)$.

\item[(iii)]  Given $1<p<\infty$, there is $K_3(p)>0$ such that
\begin{equation}
\label{aux-new-eq1}
\| e^{s\Delta}\nabla \cdot \phi\|_{ L^p(\Omega)}\leq  K_3(p)(1+s^{-\frac{1}{2}})\|\phi\|_{ L^p(\Omega)}
 \end{equation}
 for all $s>0$ and $\phi \in C^1(\bar{\Omega}, \mathbb{R}^n)$ satisfying $\phi \cdot \nu=0$  on $\partial \Omega,$ where $\nu$ is the outward normal vector to $\partial \Omega$. Consequently, for all $s>0$,  the operator $e^{s\Delta}\nabla\cdot$ possesses a uniquely determined
extension to an operator from $L^p(\Omega,\RR^n)$ into $L^p(\Omega,\RR^N)$, with norm controlled according to \eqref{aux-new-eq1}.

\item[(iv)] For given $2\le p<\infty$, there exists a positive constant $K_4(p)$ which only depends on $\Omega$ such that
$$\|\nabla {e^{t\Delta}}v\|_{L^{p}(\Omega)} \leq K_4(p)\|\nabla v\|_{L^{p}(\Omega)}\,\,  \forall\, t>0 \,\, \text{and}\,\, \forall \, v  \in W^{1,p}(\Omega).$$

\item[(v)] For given $1\le q\le p\le\infty$, there is $K_5(q,p)$ such that
$$
\|\nabla { e^{t\Delta}}v\|_{L^p}\le K_5(q,p)\Big(1+t^{-\frac{1}{2}-\frac{n}{2}\big(\frac{1}{q}-\frac{1}{p}\big)}\Big)\|v\|_{L^q}\,\, \forall\, t>0\,\,\text{and}\,\, v\in L^q(\Omega).
$$

\item[(vi)] Given $1<p<\infty,$ there is  $K_6(p)>0$ such that
$$\|\nabla e^{s\Delta}\phi\|_{ L^p(\Omega)}\leq  K_6(p)s^{-\frac{1}{2}}\|\phi\|_{ L^\infty(\Omega)} \quad \forall s>0 \,\, \text{and}\,\, \phi \in L^{\infty}(\Omega).$$
\end{itemize}
\end{lemma}

\begin{proof}
(i) First, by \cite[Lemma 1.3(i)]{Win2010}, there is $K_{1,0}(q,p)>0$ such that
\eqref{eq-estimate-000} holds for all $v\in L^q(\Omega)$ with $\int_\Omega v=0$ and with $K_1(q,p)=K_{1,0}(q,p)$.
Now for any $v\in L^q(\Omega)$,
we have
\begin{align*}
\|e^{t\Delta }v\|_p&=\| e^{t\Delta} (v-\int_\Omega v)+ e^{ t\Delta}\int_\Omega v\|_p\\
&\le \| e^{t\Delta} (v-\int_\Omega v)\|_p+\|e^{t\Delta }\int_\Omega v\|_p\\
&\le K_{1,0}(q,p) \|v-\int_\Omega v\|_q +\|\int_\Omega v\|_p\\
&\le K_{1,0}(q,p) (1+|\Omega|)\|v\|_q +|\Omega|^{\frac{1}{p}+\frac{1}{q{'}}} \|v\|_q,
\end{align*}
where $q{'}\ge 1$ is such that $\frac{1}{q}+\frac{1}{q{'}}=1$.
This implies that (i) holds for any $v\in L^q(\Omega)$ with $K_1(q,p)=K_{1,0}(q,p)(1+|\Omega|)+|\Omega|^{\frac{1}{p}+\frac{1}{q{'}}}$.

(ii) \eqref{eq-estimate-00} is equation (4.7) and \eqref{eq-estimate-01} is equation (4.8) in \cite{Win2014} respectively.

(iii) This is equation (4.12) in \cite{Win2014}.

(iv) It follows from \cite[Lemma 1.3(iii)]{Win2010}.

(v) It follows from \cite[Lemma 1.3(ii)]{Win2010}.

(vi) This is equation (4.2) in \cite{Win2014}.
\end{proof}

\smallskip
\begin{lemma}\cite[Lemma 1.1]{Win2010b}
\label{local-existence}
For any initial $(u_0,v_0) \in C(\bar{\Omega})\times W^{1,\infty}(\Omega)$  { with $u_0,v_0\geq 0$}, there exists $T_{\max}:=T_{\max}(t_0,u_0,v_0) \in (0,\infty]$ and a {unique}  classical solution $(u(x,t;t_0,u_0,v_0)$, $v(x,t;t_0,u_0,v_0))$ of \eqref{u-v-eq00} with initial condition
 $u(t_0,x)=u_0(x)$ and $v(t_0,x)=v_0(x)$ in the sense of  \eqref{initial-cond-eq}  in $\Omega \times (t_0,t_0+T_{\max})$ satisfying
\begin{equation}\label{local-existence-eq00}
\text{either} \quad T_{\max}=\infty \quad \text{or} \quad  \limsup_{t \to T_{\max}} \|u(\cdot,t_0+t;t_0,u_0,v_0)\|_{\infty}=\infty,
\end{equation}
{and for any $q>n$},
\[u \in  C^0(\bar{\Omega} \times [t_0,t_0+T_{\max})) \cap C^{2,1}(\bar{\Omega} \times (t_0,t_0+T_{\max})) \quad \text{and}  \]
\[v \in  C^0(\bar{\Omega} \times [t_0,t_0+T_{\max})) \cap C^{2,1}(\bar{\Omega} \times (t_0,t_0+T_{\max}))\cap { L^\infty_{\rm loc}}((t_0,t_0+T_{\max});W^{1,q}(\Omega)) .  \]
\end{lemma}

\begin{proof} It can be proved by the similar arguments as those in \cite[Lemma 1.1]{Win2010b}.
\end{proof}

{
\begin{lemma}
\label{dependence-on-parameters-lm}
For any $t_0 \in \mathbb{R},$ $(u_0,v_0), (\tilde u_0,\tilde v_0) \in C(\bar{\Omega})\times W^{1,\infty}(\Omega)$ with $u_0,v_0,\tilde u_0,\tilde v_0 \geq 0$,  if $(u(t),v(t)):=(u(\cdot,t;t_0,u_0,v_0),v(\cdot,t;t_0,u_0,v_0))$, $(\tilde u(t),\tilde v(t)):=(\tilde u(\cdot,t;t_0,\tilde u_0,\tilde v_0),\tilde v(\cdot,t;t_0,\tilde u_0,\tilde v_0))$ are solution of \eqref{u-v-eq00} with $(u(\cdot,t_0;t_0,u_0,v_0),v(\cdot,t_0;t_0,u_0,v_0))=(u_0,v_0)$ and $(\tilde u(\cdot,t;t_0,\tilde u_0,\tilde v_0)$, $\tilde v(\cdot,t;t_0,\tilde u_0,\tilde v_0))=(\tilde u_0,\tilde v_0)$, then there are $C_0>0$ (independent of $t,t_0,u_0,v_0$) and $C_1(t)=C_1(\sup_{t_0\le s\le t}\|u(s)\|_\infty$, $\sup_{t_0\le s\le t}\|\tilde u(s)\|_\infty$, $\sup_{t_0\le s\le t}\|v(s)\|_{W^{1,\infty}}$, $\sup_{t_0\le s\le t}\|\tilde v(s)\|_{W^{1,\infty}})>0$ such that
\begin{align}
\label{proof_dependence-lm-eq0}
&\|(u-\tilde u)(t)\|_\infty+\|(v-\tilde v)(t)\|_{W^{1,\infty}}\nonumber\\
&\le C_0   e^{-\nu_1 (t-t_0)}\big(\|u_0-\tilde u_0\|_\infty+\|v_0-\tilde v_0\|_{W^{1,\infty}}\big)\nonumber\\
&+C_1(t)\int_{t_0}^te^{-\nu_1(t-s)}(1+(\frac{t-s}{2})^{-\frac{n}{2q}})(1+\nu_2(t-s)^{-\frac{1}{2}}) \|(u-\tilde u)(s)\|_\infty ds\nonumber\\
&+C_1(t)\int_{t_0}^te^{-\nu_1(t-s)}(1+(\frac{t-s}{2})^{-\frac{n}{2q}})(1+\nu_2(t-s)^{-\frac{1}{2}}) \|(v-\tilde v)(s)\|_{W^{1,\infty}}ds
\end{align}
for any $t\in [t_0,t_0+T_{\max})$ { and $q>n$}, where $T_{\max}=\min\{T_{\max}(t_0,u_0,v_0),T_{\max}(t_0,\tilde u_0,\tilde v_0)\}$,
$\nu_1=\min\{1,\frac{\lambda}{\tau}\}$, and  $\nu_2=\max\{1,\sqrt {\frac{\tau}{2}}\}$.
\end{lemma}

\begin{proof}
Let  $\phi(t)=(u-\tilde u)(t)$ and $\varphi (t)=(v-\tilde v)(t).$ Fix $t$ such that $t_0\leq t<t_0+T_{\max}$. Then we have
\begin{align}
\label{proof_dependence-lm-eq000}
\phi(t)=&e^{-(t-t_0)}e^{{(t-t_0)}\Delta}\phi(t_0)-\chi\int_{t_0}^te^{-(t-s)}e^{{ \frac{(t-s)}{2}}\Delta}e^{{ \frac{(t-s)}{2}}\Delta}\nabla\cdot (\phi(s) \nabla v(s) )ds \nonumber\\
&-\chi\int_{t_0}^te^{-(t-s)}e^{{ \frac{(t-s)}{2}}\Delta }e^{{ \frac{(t-s)}{2}}\Delta}\nabla\cdot (\tilde u(s)  \nabla \varphi(s))ds \nonumber\\
&+\int_{t_0}^te^{-(t-s)}e^{{(t-s)}\Delta}\big[1+a_0(s)-a_1(s)(u(s)+\tilde u(s))-a_2(s)\int_\Omega u(s)\big ]\phi(s)ds
\nonumber\\
&- \int_{t_0}^te^{-(t-s)}e^{{(t-s)}\Delta} a_2(s)\tilde u(s)\int_\Omega\phi(s)ds,
\end{align}
and
\begin{equation}
\label{proof_dependence-lm-eq001}
\varphi(t)=e^{-\frac{\lambda(t-t_0)}{\tau}}e^{{ (\frac{t-t_0}{\tau})}\Delta}\varphi(t_0)+\frac{\mu}{\tau}\int_{t_0}^te^{-\frac{\lambda(t-s)}{\tau}}e^{{ (\frac{t-s}{\tau})}\Delta}\phi(s)ds.
\end{equation}
Fix some $q>n$. Thus by Lemma \ref{estimates-lm-00} (i), (iii),(iv) and (v), we get  from equations \eqref{proof_dependence-lm-eq000} and \eqref{proof_dependence-lm-eq001} respectively that
\begin{align}
\label{proof_dependence-lm-eq002}
\|\phi(t)\|_\infty&\le  e^{-(t-t_0)}\|\phi(t_0)\|_\infty\nonumber\\
& \,\, +D_0(t,t_0)\int_{t_0}^te^{-(t-s)}(1+(\frac{t-s}{2})^{-\frac{n}{2q}})(1+(\frac{t-s}{2})^{-\frac{1}{2}})\|\phi(s)\|_\infty ds \nonumber\\
&\,\, +D_1(t,t_0)\int_{t_0}^te^{-(t-s)}(1+(\frac{t-s}{2})^{-\frac{n}{2q}})(1+(\frac{t-s}{2})^{-\frac{1}{2}})\|\varphi(s)\|_{W^{1,\infty}} ds \nonumber\\
&\,\, +D_2(t,t_0)\int_{t_0}^te^{-(t-s)}\|\phi(s)\|_\infty ds,
\end{align}
and
\begin{align}
\label{proof_dependence-lm-eq003}
&\|\varphi(t)\|_{W^{1,\infty}} \nonumber\\
&\leq e^{-\frac{\lambda(t-t_0)}{\tau}}\|\varphi(t_0)\|_\infty+K_4({q}){|\Omega|^{\frac{1}{q}}}e^{-\frac{\lambda(t-t_0)}{\tau}}\|\nabla \varphi(t_0)\|_\infty\nonumber\\
& \,\, +(K_1(\infty,\infty)+K_5(\infty,\infty))\frac{\mu}{\tau}\int_{t_0}^te^{-\frac{\lambda(t-s)}{\tau}}(1+(\frac{t-s}{\tau})^{-\frac{1}{2}})\|\phi(s)\|_\infty ds,
\end{align}
where
$$
D_0(t,t_0)=|\chi|K_1(q,\infty)K_3(q)|\Omega|^{\frac{1}{q}}\sup_{t_0\leq s \leq t}\|v(s)\|_{W^{1,\infty}},
$$
$$
D_1(t,t_0)=|\chi|K_1(q,\infty)K_3(q)|\Omega|^{\frac{1}{q}}\sup_{t_0\leq s \leq t}\|\tilde u(s)\|_\infty,
$$
and
$$
D_2(t,t_0)=1+a_{0,\sup}+(a_{1,\sup}+|\Omega|a_{2,\sup})(\sup_{t_0\leq s \leq t}\| u(s)\|_\infty+\sup_{t_0\leq s \leq t}\|\tilde u(s)\|_\infty).
$$
The lemma then follows.
\end{proof}

}
\section{Global existence of bounded classical solutions}
\label{S:Global}

{ In this section, we study the global existence of classical solutions and prove Theorem \ref{thm-global-000}. We first prove a lemma.}

\begin{lemma}
\label{L1-bound}
Suppose $\inf_{t \in \mathbb{R}}\big(a_{1,\inf}(t)-{ |\Omega|(a_{2,\inf}(t))_-}\big)>0$. { Then for any $t_0 \in \mathbb{R},$ $(u_0,v_0) \in C(\bar{\Omega})\times W^{1,\infty}(\Omega)$ with $u_0,v_0 \geq 0$,  if $(u(\cdot,t;t_0,u_0,v_0),v(\cdot,t;t_0,u_0,v_0))$ is the solution of \eqref{u-v-eq00} with $(u(\cdot,t_0;t_0,u_0,v_0),v(\cdot,t_0;t_0,u_0,v_0))=(u_0,v_0)$, we have
\begin{equation}
\label{L1-bound-eq000}
0\leq \int_\Omega u(\cdot,t;t_0,u_0,v_0) \leq \max\{ \int_\Omega u_0, \tilde M_1\}:=M_0(\|u_0\|_\infty)\,\,\,\, \forall\, t_0\leq t<t_0+T_{\max}(t_0,u_0,v_0),
\end{equation}
where
\begin{equation}
\label{m1-def-eq}
\tilde M_1=\frac{|\Omega|a_{0,\sup}}{\inf_{t \in \mathbb{R}}\big(a_{1,\inf}(t)-|\Omega|(a_{2,\inf}(t))_-\big)}.
\end{equation}
Moreover if $T_{\max}(t_0,u_0,v_0)=\infty,$  then there exists $t^1(u_0)$ such that
\begin{equation}
\label{L1-bound-eq00}
0\leq \int_\Omega u(\cdot,t;t_0,u_0,v_0) \leq  M_1:=\tilde M_1+1 \quad \forall\, t\geq t_0+t^1.
\end{equation}
}
\end{lemma}
\begin{proof}
By integrating the first equation of \eqref{u-v-eq00}, we get
\begin{align}\label{L1-bound-eq01}
  \frac{d}{dt}\int_{\Omega} u =& \int_{\Omega} u\big(a_0(t,x)-a_1(t,x)u-a_2(t,x)\int_{\Omega} u\big) \nonumber\\
  \leq & \int_{\Omega} u\big(a_{0,\sup}-a_{1,\inf}(t)u+(a_{2,\inf}(t))_-\int_{\Omega} u\big) \nonumber\\
  \leq & \int_{\Omega} u\big(a_{0,\sup}-\frac{a_{1,\inf}(t)-|\Omega|(a_{2,\inf}(t))_-}{|\Omega|}\int_{\Omega} u\big)\nonumber\\
  \leq & \int_{\Omega} u\big(a_{0,\sup}-\frac{\inf_{t \in \mathbb{R}}\big(a_{1,\inf}(t)-|\Omega|(a_{2,\inf}(t))_-\big)}{|\Omega|}\int_{\Omega} u\big).
\end{align}
{Then \eqref{L1-bound-eq000} follows from \eqref{L1-bound-eq01} and the comparison principle for ordinary differential equations. Furthermore if $T_{\max}(t_0,u_0,v_0)=\infty,$ we get   $\int_\Omega u(\cdot,t;t_0,u_0,v_0)\leq y(t;t_0,|\Omega|\|u_0\|_\infty)$ for all $t\geq t_0$
with $y(t;t_0,|\Omega|\|u_0\|_\infty)$ satisfying the following ordinary differential equation,
\begin{equation}
 \label{L1-bound-eq01}
y'=y\big(a_{0,\sup}-\frac{\inf_{t \in \mathbb{R}}\big(a_{1,\inf}(t)-|\Omega|(a_{2,\inf}(t))_-\big)}{|\Omega|}y\big),\quad t>t_0,
\end{equation}
with initial $y(t_0)=|\Omega|\|u_0\|_\infty.$
{ This implies that} there exists $t^1_\epsilon=t^1(u_0)$
\[
y(t;t_0,|\Omega|\|u_0\|_\infty) \leq M_1  \quad \forall t\geq t_0+t^1.
\]
Thus equation \eqref{L1-bound-eq00} follows.}
\end{proof}

Next, we prove Theorem \ref{thm-global-000} under the assumption of {\bf (H1)}.

\begin{proof}[Proof of Theorem \ref{thm-global-000} with the assumption {\bf (H1)}]

 {Assume that {\bf (H1)} holds. Theorem \ref{thm-global-000}  can then  be proved by properly  modifying arguments of \cite[Theorem 2.2]{JZYYLGBXZ2018}.
 For completeness, we provide a proof in the following.}

  We divide  the proof in six steps. For simplicity in notation, we put $T_{\max}=T_{\max}(t_0,u_0,v_0)$, and
  $$
  (u(t),v(t))=(u(\cdot,t;t_0,u_0,v_0),v(\cdot,t;t_0,u_0,v_0)).
  $$
Note that, by {\bf (H1)},
$$a_{1,\inf}>\inf_{q>\max\{1,\frac{n}{2}\}}\Big(\frac{q-1}{q}[C_{q+1}]^{\frac{1}{q+1}} \mu^{\frac{1}{q+1}}\Big)|\chi|.
$$
Hence there is $\gamma>1$ such that
$$a_{1,\inf}>\Big(\frac{\gamma-1}{\gamma}[C_{\gamma+1}]^{\frac{1}{\gamma+1}} \mu^{\frac{1}{\gamma+1}}\Big)|\chi|
$$
and hence
$$
|\chi | \mu^{\frac{1}{\gamma+1}}C^{\frac{1}{\gamma+1}}_{\gamma+1}>\big(|\chi|\mu^{\frac{1}{\gamma+1}} C^{\frac{1}{\gamma+1}}_{\gamma+1}-a_{1,\inf}\big)\gamma.
$$
Therefore, there is $\gamma>1$ such that
$$
\gamma \in \Big(1,\frac{|\chi | \mu^{\frac{1}{\gamma+1}}C^{\frac{1}{\gamma+1}}_{\gamma+1}}{\big(|\chi|\mu^{\frac{1}{\gamma+1}} C^{\frac{1}{\gamma+1}}_{\gamma+1}-a_{1,\inf}\big)_+} \Big).
$$

\smallskip

\noindent{\bf Step 1.} {\it  In this step, we prove that  for any $\gamma>1$ satisfying that $\gamma \in \Big(1,\frac{|\chi | \mu^{\frac{1}{\gamma+1}}C^{\frac{1}{\gamma+1}}_{\gamma+1}}{\big(|\chi|\mu^{\frac{1}{\gamma+1}} C^{\frac{1}{\gamma+1}}_{\gamma+1}-a_{1,\inf}\big)_+} \Big),$  there is $C=C(\gamma,u_0, v_0,a_0,a_1,a_2,|\Omega|)$  such that}
\begin{equation}
\label{new-global-eq1}
 \int_{\Omega}u^{\gamma} (t)     \leq C   \quad \forall\,\, t \in [t_0, t_0+T_{\max}).
 \end{equation}

 First, by multiplying the  first equation of  $\eqref{u-v-eq00}$  by $u^{\gamma-1}(t)$ and integrating with respect to $x$ over $\Omega,$  we have for $t\in (t_0,t_0+T_{\max})$ that
      \begin{align*}
      \frac{1}{\gamma}\frac{d}{dt}\int_{\Omega}u^{\gamma}(t)+\frac{4(\gamma-1)}{\gamma^2}\int_{\Omega}  |\nabla u^{\frac{\gamma}{2}}(t)|^2=&(\gamma-1)\chi\int_{\Omega}u^{\gamma-1}(t)\nabla u (t)\cdot \nabla v(t) \\   &+\int_{\Omega}u^{\gamma}(t)\Big[a_0(t,\cdot)-a_1(t,\cdot)u(t)-a_2(t,\cdot)\int_{\Omega}u(t) \Big].
              \end{align*}
By Lemma \ref{L1-bound}, we have
{
\[\int_{\Omega} u(\cdot,t;t_0,u_0,v_0) \leq \max\bigg\{\int_{\Omega}u_0, \frac{|\Omega|a_{0,\sup}}{\inf_{t \in \mathbb{R}}\big(a_{1,\inf}(t)-|\Omega|(a_{2,\inf}(t))_-\big)}\bigg\}:=M_0
\]
for all $ t \in [t_0, t_0+T_{\max}).$
}
Thus
{
 \begin{align}
 \label{proof-global-2-eq00}
   &   \frac{1}{\gamma}\frac{d}{dt}\int_{\Omega}u^{\gamma}(t)+\frac{4(\gamma-1)}{\gamma^2}\int_{\Omega}  |\nabla u^{\frac{\gamma}{2}}(t)|^2\nonumber\\
 &\leq(\gamma-1)\chi\int_{\Omega}u^{\gamma-1}(t)\nabla u (t)\cdot \nabla v(t) \nonumber\\  
  &\,\,\,\,  +\int_{\Omega}u^{\gamma}(t)\Big[a_{0,\sup}-a_{1,\inf}u(t)+(a_{2,\inf})_-M_0\Big]\nonumber\\
      &=-\frac{\chi(\gamma-1)}{\gamma}\int_{\Omega}u^\gamma(t) \Delta v(t) -\frac{\gamma+1}{\gamma}\int_{\Omega}u^\gamma(t)\nonumber\\  
       &\,\,\,\, +\int_{\Omega}u^{\gamma}(t)\Big[a_{0,\sup}+\frac{\gamma+1}{\gamma}+(a_{2,\inf})_-M_0-a_{1,\inf}u(t)\Big].                      \end{align}
}
 Let $\epsilon>0$. By Young's inequality with $p=\frac{\gamma+1}{\gamma}$ and $q=\gamma+1$, we get
 \begin{align*}
 &\int_{\Omega}\Big[a_{0,\sup}+\frac{\gamma+1}{\gamma}+(a_{2,\inf})_-M_0\Big]u^{\gamma}(t)\\
 &\leq \epsilon\int_{\Omega}u^{\gamma+1}+\underbrace{\frac{1}{\gamma+1}\big[\frac{\gamma+1}{\gamma}\epsilon\big]^{-\gamma}[a_{0,\sup}
 +\frac{\gamma+1}{\gamma}+(a_{2,\inf})_-M_0\Big]^{\gamma+1}|\Omega|}_{C_1(\epsilon,a_0,a_1,a_2\gamma,\int_\Omega u_0):=C_1}.
 \end{align*}

 By combining this last equation with equation \eqref{proof-global-2-eq00}, we get
  \begin{align}
  \label{proof-global-2-eq01}
      \frac{1}{\gamma}\frac{d}{dt}\int_{\Omega}u^{\gamma}(t)+\frac{4(\gamma-1)}{\gamma^2}\int_{\Omega}  |\nabla u^{\frac{\gamma}{2}}(t)|^2\leq &\frac{|\chi|(\gamma-1)}{\gamma}\int_{\Omega}u^\gamma(t) |\Delta v(t)| -\frac{\gamma+1}{\gamma}\int_{\Omega}u^\gamma(t)\nonumber\\   &+(\epsilon-a_{1,\inf})\int_{\Omega}u^{\gamma+1}(t) + C_1.                      \end{align}
  Let $r>0$. By Young's inequality with $p=\frac{\gamma+1}{\gamma}$ and $q=\gamma+1$ again, we get
  \[\frac{|\chi|(\gamma-1)}{\gamma}\int_{\Omega}u^\gamma(t) |\Delta v(t)|\leq r\int_{\Omega}u^{\gamma+1}(t)+\underbrace{\frac{1}{\gamma+1}\big[\frac{\gamma+1}{\gamma}\big]^{-\gamma}\big[\frac{\gamma-1}{\gamma}\big]^{\gamma+1} }_{A_\gamma} r^{-\gamma}|\chi|^{\gamma+1}\int_{\Omega}|\Delta v(t)|^{\gamma+1}. \]
  By combining this last equation with equation \eqref{proof-global-2-eq01}, we get
  \begin{align}
  \label{proof-global-2-eq02}
      \frac{1}{\gamma}\frac{d}{dt}\int_{\Omega}u^{\gamma}(t)+\frac{4(\gamma-1)}{\gamma^2}\int_{\Omega}  |\nabla u^{\frac{\gamma}{2}}(t)|^2\leq &-\frac{\gamma+1}{\gamma}\int_{\Omega}u^\gamma(t)+(\epsilon+ r-a_{1,\inf})\int_{\Omega}u^{\gamma+1}(t)\nonumber\\
         &+A_\gamma r^{-\gamma}|\chi|^{\gamma+1}\int_{\Omega}|\Delta v(t)|^{\gamma+1} + C_1.                      \end{align}
 Let  $s_0 \in (0,T_{\max})$  be fixed.  By Lemma \ref{local-existence},  there exists a positive constant $K=K(u_0,v_0)$ such that
  \begin{equation}\label{proof-global-2-eq03}
    \|u(t)\|_{\infty}\leq K,\,\, \|v(t)\|_{\infty}\leq K\,\,  \forall\, t \in (t_0,t_0+s_0],  \,\text{and}\,\, \|\Delta v(t_0+s_0)\|_{ \gamma+1}\leq K.
  \end{equation}

  Next let $y$ be the solution of the following ordinary differential equation,
  \begin{equation}
 \label{proof-global-2-ode00}
\begin{cases}
y'=-(\gamma+1)y+\gamma f(t),t\in (t_0+s_0, t_0+T_{\max})\cr
y(s_0)=\|u(t_0+s_0)\|_\infty,
\end{cases}
 \end{equation}
 where
 \[f(t)=(\epsilon+ r-a_{1,\inf})\int_{\Omega}u^{\gamma+1}(t)+A_\gamma r^{-\gamma}|\chi|^{\gamma+1}\int_{\Omega}|\Delta v(t)|^{\gamma+1} + C_1. \]
 Then, by equations \eqref{proof-global-2-eq02} and \eqref{proof-global-2-ode00}, the comparison principle for parabolic equations,  and variation of constant formula, we get
 \[\int_{\Omega}u^{\gamma}(t)\leq y(t)=e^{-(\gamma+1)(t-t_0-s_0)}\|u(t_0+s_0)\|_\infty+\gamma\int_{t_0+s_0}^{t}e^{-(\gamma+1)(t-s)}f(s)ds\]
  for all $t_0+s_0\le t<t_0+T_{\max}$.
 This is equivalent to
 \begin{align}
  \label{proof-global-2-eq04}
      \frac{1}{\gamma}\int_{\Omega}u^{\gamma}(t)\leq &\frac{1}{\gamma}e^{-(\gamma+1)(t-t_0-s_0)}\|u(t_0+s_0)\|_\infty+(\epsilon+ r-a_{1,\inf})\int_{t_0+s_0}^{t}e^{-(\gamma+1)(t-s)}\int_{\Omega}u^{\gamma+1}(s)ds\nonumber\\
         &+A_\gamma r^{-\gamma}|\chi|^{\gamma+1}\int_{t_0+s_0}^{t}e^{-(\gamma+1)(t-s)}\int_{\Omega}|\Delta v(s)|^{\gamma+1}ds + C_1\int_{t_0+s_0}^{t}e^{-(\gamma+1)(t-s)}ds  \nonumber\\
         \leq &(\epsilon+ r-a_{1,\inf})\int_{t_0+s_0}^{t}e^{-(\gamma+1)(t-s)}\int_{\Omega}u^{\gamma+1}(s)ds\nonumber\\
         &+A_\gamma r^{-\gamma}|\chi|^{\gamma+1}\int_{t_0+s_0}^{t}e^{-(\gamma+1)(t-s)}\int_{\Omega}|\Delta v(s)|^{\gamma+1}ds\nonumber\\
          & + \frac{K}{\gamma}e^{-(\gamma+1)(t-t_0-s_0)}+\frac{C_1}{\gamma+1}  \end{align}
for $t_0+s_0\le t<t_0+T_{\max}$.

 {Now, by Lemma \ref{prelimiaries-lm-00},}
 \begin{align}
\label{proof-global-2-eq05}
   &\int_{t_0+s_0}^{t}e^{-(\gamma+1)(t-s)}\|\Delta v(\cdot,s)\|^{\gamma+1}_{L^{\gamma+1}(\Omega)} ds \nonumber\\
  &\leq { C_{\gamma+1} \mu} \int_{t_0+s_0}^{t}e^{-(\gamma+1)(t-s)}\|u(\cdot,s)\|^{\gamma+1}_{L^{\gamma+1}(\Omega)}ds\nonumber\\
&+ C_{\gamma+1}e^{-(\gamma+1) (t-t_0-s_0)}\big(\|v(\cdot,t_0+s_0)\|^{\gamma+1}_{L^{\gamma+1}(\Omega)}+\|\Delta v(\cdot,t_0+s_0)\|^{\gamma+1}_{L^{\gamma+1}(\Omega)}\big)
\end{align}
for $t_0+s_0\le t<t_0+T_{\max}$.
Combining equations \eqref{proof-global-2-eq04} and \eqref{proof-global-2-eq05}, we get using in addition \eqref{proof-global-2-eq03}
\begin{align}
  \label{proof-global-2-eq06}
      \frac{1}{\gamma}\int_{\Omega}u^{\gamma}(t)
         \leq &{ (\epsilon+A_\gamma C_{\gamma+1}r^{-\gamma}|\chi|^{\gamma+1}\mu}+ r-a_{1,\inf})\int_{t_0+s_0}^{t}e^{-(\gamma+1)(t-s)}\int_{\Omega}u^{\gamma+1}(s)ds\nonumber\\
         & + e^{-(\gamma+1)(t-t_0-s_0)}\Big(\frac{K}{\gamma}  +2A_\gamma r^{-\gamma}|\chi|^{\gamma+1}C_{\gamma+1}K^{\gamma+1} \Big)+ \frac{C_1}{\gamma+1} \end{align}
         for $t_0+s_0\le t<t_0+T_{\max}$.

 We claim that
\begin{equation}\label{proof-global-2-eq07}
 \min_{r>0} \big(\underbrace{A_\gamma C_{\gamma+1}r^{-\gamma}|\chi|^{\gamma+1}\mu+ r}_{H(r)}\big)=H\big(\underbrace{(A_\gamma C_{\gamma+1}\gamma)^{\frac{1}{\gamma+1}}|\chi| \mu^{\frac{1}{\gamma+1}}}_{r_0}\big)=\frac{\gamma-1}{\gamma}C^{\frac{1}{\gamma+1}}_{\gamma+1}|\chi| \mu^{\frac{1}{\gamma+1}}.
\end{equation}
Indeed, we have $H'(r)=-\gamma A_\gamma C_{\gamma+1} r^{-\gamma-1}|\chi|^{\gamma+1}\mu+1.$ Thus
{  $$H'(r)=0 \iff r=\big(\gamma A_\gamma C_{\gamma+1}\mu\big)^{\frac{1}{\gamma+1}}|\chi|.$$ }
Furthermore, $H^{''}(r)=\gamma(\gamma+1) A_\gamma C_{\gamma+1} r^{-\gamma-2}|\chi|^{\gamma+1}>0,\,\, \forall r>0.$ Thus $\min_{r>0}H(r)=H(r_0),$ which is given by
\begin{align}
\label{min-H-eq00}
H(r_0)=&A_\gamma C_{\gamma+1}\big(\big(\gamma A_\gamma C_{\gamma+1}|\chi|^{\gamma+1}\mu\big)^{\frac{1}{\gamma+1}}\big)^{-\gamma}|\chi|^{\gamma+1}\mu+\big(\gamma A_\gamma C_{\gamma+1}|\chi|^{\gamma+1}\mu\big)^{\frac{1}{\gamma+1}}\nonumber\\
=&A_\gamma^{\frac{1}{\gamma+1}} C_{\gamma+1}^{\frac{1}{\gamma+1}}|\chi| \mu^{\frac{1}{\gamma+1}}\big(\gamma^{-\frac{\gamma}{\gamma+1}}+\gamma^\frac{1}{\gamma+1}\big)\nonumber\\
=&A_\gamma^{\frac{1}{\gamma+1}} C_{\gamma+1}^{\frac{1}{\gamma+1}}|\chi| \mu^{\frac{1}{\gamma+1}}\gamma^{-\frac{\gamma}{\gamma+1}}(1+\gamma).
\end{align}
Note that $A_\gamma=\frac{1}{\gamma+1}\big[\frac{\gamma+1}{\gamma}\big]^{-\gamma}\big[\frac{\gamma-1}{\gamma}\big]^{\gamma+1}.$ Thus
\begin{align}
\label{min-H-eq01}
A_\gamma^{\frac{1}{\gamma+1}}=&\big[\frac{1}{\gamma+1}\big]^\frac{1}{\gamma+1}\big[\frac{\gamma+1}{\gamma}\big]^{-\frac{\gamma}{\gamma+1}}\big(\frac{\gamma-1}{\gamma}\big)\nonumber\\
=&\big[\gamma+1\big]^{-\frac{1}{\gamma+1}}\big[\frac{\gamma+1}{\gamma}\big]^{-\frac{\gamma}{\gamma+1}}\big(\frac{\gamma-1}{\gamma}\big)\nonumber\\
=&\big[\gamma+1\big]^{-1}\gamma^{\frac{\gamma}{\gamma+1}}\big(\frac{\gamma-1}{\gamma}\big).
\end{align}
From equations \eqref{min-H-eq00} and \eqref{min-H-eq01}, we get
\begin{align*}
H(r_0)=&A_\gamma^{\frac{1}{\gamma+1}}C_{\gamma+1}^{\frac{1}{\gamma+1}}|\chi| \mu^{\frac{1}{\gamma+1}}\gamma^{-\frac{\gamma}{\gamma+1}}(1+\gamma)\\
=&\big[\gamma+1\big]^{-1}\gamma^{\frac{\gamma}{\gamma+1}}\big(\frac{\gamma-1}{\gamma}\big)C_{\gamma+1}^{\frac{1}{\gamma+1}}|\chi| \mu^{\frac{1}{\gamma+1}}\gamma^{-\frac{\gamma}{\gamma+1}}(1+\gamma)\\
=&\big(\frac{\gamma-1}{\gamma}\big)C_{\gamma+1}^{\frac{1}{\gamma+1}}|\chi| \mu^{\frac{1}{\gamma+1}},
\end{align*}
and \eqref{proof-global-2-eq07} follows.

Finally, combining  equations \eqref{proof-global-2-eq06} and \eqref{proof-global-2-eq07}, we get
\begin{align}
  \label{proof-global-2-eq08}
      \frac{1}{\gamma}\int_{\Omega}u^{\gamma}(t)\leq
 &(\epsilon+{ \frac{\gamma-1}{\gamma}C^{\frac{1}{\gamma+1}}_{\gamma+1}|\chi|\mu^{\frac{1}{\gamma+1}}}-a_{1,\inf})\int_{t_0+s_0}^{t}e^{-(\gamma+1)(t-s)}\int_{\Omega}u^{\gamma+1}(s)ds\nonumber\\
         & + e^{-(\gamma+1)(t-t_0-s_0)}\Big(\frac{K}{\gamma}  +2A_\gamma r_0^{-\gamma}|\chi|^{\gamma+1}C_{\gamma+1}K^{\gamma+1}|\Omega| \Big)+ \frac{C_1}{\gamma+1}, \end{align}
         where  $r_0=\big(\gamma A_\gamma C_{\gamma+1}\mu\big)^{\frac{1}{\gamma+1}}|\chi|$.
{ Since $\gamma \in \Big(1, \frac{|\chi| \mu^{\frac{1}{\gamma+1}} C^{\frac{1}{\gamma+1}}_{\gamma+1}}{\big(|\chi| \mu^{\frac{1}{\gamma+1}} C^{\frac{1}{\gamma+1}}_{\gamma+1}-a_{1,\inf}\big)_+} \Big),$  we have $a_{1,\inf}>\frac{\gamma-1}{\gamma}C^{\frac{1}{\gamma+1}}_{\gamma+1}|\chi|\mu^{\frac{1}{\gamma+1}}.$ By choosing $\epsilon<a_{1,\inf}-\frac{\gamma-1}{\gamma}C^{\frac{1}{\gamma+1}}_{\gamma+1}|\chi|\mu^{\frac{1}{\gamma+1}},$  we get from \eqref{proof-global-2-eq08} for $t \in (t_0+s_0,T_{\max})$ that
\begin{equation}\label{proof-global-2-eq09}
      \frac{1}{\gamma}\int_{\Omega}u^{\gamma}(t)
         \leq  e^{-(\gamma+1)(t-t_0-s_0)}\Big(\frac{K}{\gamma}  +2A_\gamma r_0^{-\gamma}|\chi|^{\gamma+1}C_{\gamma+1}K^{\gamma+1}|\Omega| \Big)+ \frac{C_1}{\gamma+1}. \end{equation}
The proof of Step 1 follows from \eqref{proof-global-2-eq03} and \eqref{proof-global-2-eq09}.
}

\smallskip

\noindent{\bf Step 2.} {\it Let $q_0>\max\{1,\frac{n}{2}\}$ be such that   $a_{1,\inf}>\frac{q_0-1}{q_0}|\chi| [C_{q_0+1}{ \mu}]^{\frac{1}{q_0+1}}.$ In this step, we prove that for any $q \in \left[1,\frac{nq_0}{(n-q_0)_+}\right),$   there exists a constant $C=C(q_0,q,u_0,v_0,a_0,a_1,a_2,|\Omega|)$  such that}
\begin{equation}
\label{new-global-step2-eq00}
 \|\nabla v(t)\|_q    \leq C   \quad \forall\,\, t \in (t_0, t_0+T_{\max}).
 \end{equation}

First, by  Step 1,  there is $C=C(q_0, u_0,v_0, a_i,|\Omega|)$  such that
\begin{equation}
\label{new-global-step2-eq01}
 \int_{\Omega}u^{q_0} (t)     \leq C   \quad \forall\,\, t \in [t_0, t_0+T_{\max}).
 \end{equation}
Next, by the second equation in \eqref{u-v-eq00} and the variation of constant formula, we have  for all $t\in (t_0,t_0+T_{\max})$ that
\[v(t)=e^{{ {(\frac{t-t_0}{\tau})}}(\Delta-\lambda I)}v_0+\frac{\mu}{\tau}\int_{t_0}^{t}e^{{ (\frac{t-s}{\tau})}(\Delta-\lambda I)}u(s)ds.\]
By Lemma \ref{estimates-lm-00}(iv) and (v),
 we have
\begin{align}\label{new-global-step2-eq03b}
  \|\nabla v(t)\|_q  \leq &  \| \nabla e^{{{(\frac{t-t_0}{\tau})}}(\Delta-\lambda I)}v_0\|_q +\frac{\mu}{\tau}\int_{t_0}^{t}\|\nabla e^{{ (\frac{t-s}{\tau})}(\Delta-\lambda I)}u(s)\|_q \nonumber\\
    \leq & K_4(q) e^{-\lambda \big(\frac{t-t_0}{\tau}\big)}\|v_0\|_{W^{1,\infty}} \nonumber \\
   & + K_5(q_0,q)\int_{t_0}^{t}\big(1+(\frac{t-s}{\tau})^{-\frac{1}{2}-\frac{n}{2}(\frac{1}{q_0}-\frac{1}{q})_+}\big)e^{-\lambda(\frac{t-s}{\tau})}\|u(s)\|_{q_0}ds \nonumber \\
   \leq &  K_4(q) e^{-\lambda \big(\frac{t-t_0}{\tau}\big)}\|v_0\|_{W^{1,\infty}}\nonumber\\
   & + K_5(q_0,q) \tau \sup_{t\in [t_0,t_0+T_{\max})}\|u(t)\|_{L^{q_0}} \int_{0}^{\infty}\big(1+s^{-\frac{1}{2}-\frac{n}{2}(\frac{1}{q_0}-\frac{1}{q})_+}\big)e^{-\lambda s} ds,
\end{align}
for each $ t\in (t_0,t_0+T_{\max}),$ which is finite provided that $\frac{1}{2}+\frac{n}{2}(\frac{1}{q_0}-\frac{1}{q})_+<1$ which is equivalent to $q< \frac{nq_0}{(n-q_0)_+}.$ Thus Step 2 follows from  \eqref{new-global-step2-eq03b}.

\medskip

\noindent{\bf Step 3.} {\it Let $q_0$ be given as in Step 2. In this step, we prove that  for any { $ \gamma \geq 1$}, there is $C=C(\gamma, u_0,v_0,a_0,a_1,a_2|\Omega|)$ such that}
 \begin{equation}
 \label{new-global-eq3}
 \int_{\Omega}u^{\gamma}(t) \leq C \quad \forall t \in [t_0, t_0+T_{\max}).
 \end{equation}

 First, note that  $q_0<\frac{nq_0}{2(n-q_0)_+}.$  By Step 1 and Step 2, we have
 $$
 \sup_{t\in [t_0,t_0+T_{\max})} \int_{\Omega}u^{q_0}(t)<\infty
 $$
 and
 $$
  \sup_{t\in [t_0,t_0+T_{\max})} \int_{\Omega}|\nabla v|^{2q_0}(t) <\infty.
$$

{ Furthermore if $\gamma \leq q_0,$ by the continuous inclusion of $L^{q_0}(\Omega)$ into $L^{\gamma}(\Omega)$, there exists a positive constant $C_0$ depending only on $\Omega, n, q_0$ and $\gamma$ such that
\[ \|u(t)\|_{L^\gamma }\leq C_0\|u(t)\|_{L^{q_0}},\]
 and \eqref{new-global-eq3} fellows.
}

Next { suppose $\gamma>q_0$}. By the arguments in Step 1, we get
 \begin{align}
 \label{proof-global-step3-eq00}
      \frac{1}{\gamma}\frac{d}{dt}\int_{\Omega}u^{\gamma}(t)+\frac{4(\gamma-1)}{\gamma^2}\int_{\Omega}  |\nabla u^{\frac{\gamma}{2}}(t)|^2\leq &(\gamma-1)\chi\int_{\Omega}u^{\gamma-1}(t)\nabla u (t)\cdot \nabla v(t) \nonumber\\   &+\int_{\Omega}u^{\gamma}(t)\Big[a_{0,\sup}-a_{1,\inf}u(t)+(a_{2,\inf})_-M_0\Big].                      \end{align}
By Young's inequality we get
 \[ { (a_{0,\sup}+(a_{2,\inf})_-M_0)} \int_{\Omega}u^{\gamma}(t) \leq \frac{a_{1,\inf}}{2}\int_{\Omega}u^{\gamma+1}(t)+C(\gamma,a_i, M_0,|\Omega|) .   \]
This together with \eqref{proof-global-step3-eq00} implies that
 \begin{align}\label{proof-global-step3-eq01}
&\frac{1}{\gamma}\frac{d}{dt}\int_{\Omega}u^{\gamma}(t)+\frac{4(\gamma-1)}{\gamma^2}\int_{\Omega}  |\nabla u^{\frac{\gamma}{2}}(t)|^2 \nonumber\\
&\leq (\gamma-1)\chi\int_{\Omega}u^{\gamma-1}(t)\nabla u (t)\cdot \nabla v(t)- \frac{a_{1,\inf}}{2}\int_{\Omega}u^{\gamma+1}(t)+C(\gamma,a_i,M_0,|\Omega|).  \end{align}
By Young's inequality again, we have
 \begin{align*}
 &(\gamma-1)\chi\int_{\Omega}u^{\gamma-1}(t)\nabla u (t)\cdot \nabla v(t)\\
 &\leq \frac{\gamma-1}{2}\int_{\Omega}u^{\gamma-2}(t)|\nabla u (t)|^2+\frac{\chi^2(\gamma-1)}{2}\int_{\Omega}u^{\gamma}(t)|\nabla v (t)|^2\\
 &=\frac{2(\gamma-1)}{\gamma^2}\int_{\Omega}  |\nabla u^{\frac{\gamma}{2}}(t)|^2+\frac{\chi^2(\gamma-1)}{2}\int_{\Omega}u^{\gamma}(t)|\nabla v (t)|^2.
 \end{align*}
 This together with \eqref{proof-global-step3-eq01} implies that
  \begin{align}\label{proof-global-step3-eq02}
&\frac{1}{\gamma}\frac{d}{dt}\int_{\Omega}u^{\gamma}(t)+\frac{2(\gamma-1)}{\gamma^2}\int_{\Omega}  |\nabla u^{\frac{\gamma}{2}}(t)|^2 \nonumber\\
&\leq \frac{\chi^2(\gamma-1)}{2}\int_{\Omega}u^{\gamma}(t)|\nabla v (t)|^2- \frac{a_{1,\inf}}{2}\int_{\Omega}u^{\gamma+1}(t)+C(\gamma,a_i,M_0,|\Omega|).  \end{align}
By Holder's inequality, we have
\[\frac{\chi^2(\gamma-1)}{2}\int_{\Omega}u^{\gamma}(t)|\nabla v (t)|^2\leq \frac{\chi^2(\gamma-1)}{2}\Big(\int_{\Omega}u^{\frac{\gamma q_0}{q_0-1}}(t)\Big)^{\frac{q_0-1}{q_0}}\Big(\int_{\Omega}|\nabla v (t)|^{ 2q_0}\Big)^{\frac{1}{q_0}}.\]
By Gagliardo-Nirenberg inequality, there exists  a positive constant $C_0$ depending on the domain $\Omega$ and $\gamma$ such that
\begin{eqnarray*}
 \|u^{\frac{\gamma}{2}}\|^2_{L^{\frac{2q_0}{q_0-1}}} &\leq &C_0 \|\nabla u^{\frac{\gamma}{2}}(t)\|^{2a}_{L^2}  \|u^{\frac{\gamma}{2}}(t)\|_{L^{\frac{2 q_0}{\gamma}}}^{2(1-a)}+C_0\|u^{\frac{\gamma}{2}}(t)\|_{L^{\frac{2 q_0}{\gamma}}}^{2},
\end{eqnarray*}
where $a=\frac{\frac{n\gamma}{2q_0}-\frac{n(q_0-1)}{2q_0}}{1+\frac{n}{2}(\frac{\gamma}{q_0}-1)}.$
Since $\frac{n}{2}<q_0<\gamma,$ we have $0<a<1.$
By applying Young's Inequality, we get for any $\epsilon>0$
 \begin{align*}
&{ C_0 \|u^{\frac{\gamma}{2}}(t)\|_{L^{\frac{2 q_0}{\gamma}}}^{2(1-a)} \frac{\chi^2(\gamma-1)}{2}\Big(\int_{\Omega}|\nabla v (t)|^{ 2q_0}\Big)^{\frac{1}{q_0}}} \|\nabla u^{\frac{\gamma}{2}}(t)\|^{2a}_{L^2}\\
&\leq \epsilon\|\nabla u^{\frac{\gamma}{2}}(t)\|^2_{L^2}+{ C(\epsilon,\gamma,\sup_{t\in[t_0,t_0+T_{|max})}\|u(t)\|_{L^{q_0}},\sup_{t\in [t_0,t_0+T_{\max})}\|\nabla v\|_{L^{2q_0}},a_i,|\Omega|)}.
 \end{align*}
Put
$$\sup_t \|u(t)\|_{L^{q_0}}=\sup_{t\in[t_0,t_0+T_{|max})}\|u(t)\|_{L^{q_0}}\quad {\rm and}\quad \sup_t \|\nabla v\|_{L^{2q_0}}=\sup_{t\in [t_0,t_0+T_{\max})}\|\nabla v\|_{L^{2q_0}}.
$$
Then
\[\frac{\chi^2(\gamma-1)}{2}\int_{\Omega}u^{\gamma}(t)|\nabla v (t)|^2\leq  \epsilon\|\nabla u^{\frac{\gamma}{2}}(t)\|^2_{L^2}+C(\epsilon,\gamma,\sup_{t}\|u(t)\|_{L^{q_0}},\sup_{t}\|\nabla v\|_{L^{2q_0}},a_0,a_1,a_2|\Omega|). \]
It then follows from equation \eqref{proof-global-step3-eq02} that
\begin{align*}
&\frac{1}{\gamma}\frac{d}{dt}\int_{\Omega}u^{\gamma}(t)+\frac{2(\gamma-1)}{\gamma^2}\int_{\Omega}  |\nabla u^{\frac{\gamma}{2}}(t)|^2 \nonumber\\
&\leq \epsilon\|\nabla u^{\frac{\gamma}{2}}(t)\|^2_{L^2}- \frac{a_{1,\inf}}{2}\int_{\Omega}u^{\gamma+1}(t+C(\epsilon,\gamma,\sup_t\|u(t)\|_{L^{q_0}},\sup_t\|\nabla v\|_{L^{2q_0}},a_0,a_1,a_2,|\Omega|).  \end{align*}
Taking $\epsilon=\frac{2(\gamma-1)}{\gamma^2}$ in this last equation, we get
\begin{align}\label{proof-global-step3-eq03}
& \frac{1}{\gamma}\frac{d}{dt}\int_{\Omega}u^{\gamma}(t) \nonumber\\
&\leq - \frac{a_{1,\inf}}{2}\int_{\Omega}u^{\gamma+1}(t)+C(\epsilon,\gamma,\sup_t\|u(t)\|_{L^{q_0}},\sup_t\|\nabla v\|_{L^{2q_0}},a_0,a_1,a_2,|\Omega|)\nonumber\\
&\leq - \frac{a_{1,\inf}}{2|\Omega|^{\frac{1}{\gamma}}}\Big(\int_{\Omega}u^\gamma(t)\Big)^{\frac{\gamma+1}{\gamma}}
+C(\epsilon,\gamma,\sup_t\|u(t)\|_{L^{q_0}},\sup_t\|\nabla v\|_{L^{2q_0}},a_0,a_1,a_2,|\Omega|).  \end{align}
\eqref{new-global-eq3} then  follows.

\smallskip

\noindent{\bf Step 4.} {\it  In this step, we prove that for any $q\ge 1$,  there exists  $C=C(q, u_0,v_0,a_0,a_1,a_2,|\Omega|)$  such that}
\begin{equation}
\label{new-global-step4-eq00}
 \|\nabla v(t)\|_q    \leq C   \quad \forall\,\, t \in (t_0, t_0+T_{\max}).
 \end{equation}

 By the arguments in Step 2, we have
\begin{align}\label{aux-new-eq2}
  \|\nabla v(t)\|_q  \leq &  \| \nabla e^{{(\frac{t-t_0}{\tau})}(\Delta-\lambda I)}v_0\|_q +\frac{\mu}{\tau}\int_{t_0}^{t}\|\nabla e^{{ (\frac{t-s}{\tau})}(\Delta-\lambda I)}u(s)\|_q \nonumber\\
    \leq & K_4(q) e^{-\lambda \big(\frac{t-t_0}{\tau}\big)}\|v_0\|_{W^{1,\infty}} \nonumber \\
   & + K_5(q,q)\int_{t_0}^{t}\big(1+(\frac{t-s}{\tau})^{-\frac{1}{2}}\big)e^{-\lambda(\frac{t-s}{\tau})}\|u(s)\|_{q}ds \nonumber \\
   \leq &  K_4(q) e^{-\lambda \big(\frac{t-t_0}{\tau}\big)}\|v_0\|_{W^{1,\infty}}\nonumber\\
   & + K_5(q,q) \tau \sup_{t\in [t_0,t_0+T_{\max})}\|u(t)\|_{L^q} \int_{0}^{\infty}\big(1+s^{-\frac{1}{2}}\big)e^{-\lambda s} ds
\end{align}
for each $ t\in (t_0,t_0+T_{\max})$. \eqref{new-global-step4-eq00} then follows.

\medskip

\noindent{\bf Step 5.} {\it Choose $p>n$ and $p_1>p>p_2$ such that $\frac{1}{p}=\frac{1}{p_1}+\frac{1}{p_2}$.
  In this sept, we prove that  there is $ C=C(u_0,v_0)$  such that
\begin{equation}
\label{new-global-eq4}
\|u(t)\|_{C^0(\bar\Omega)}+ \|v(t)\|_{C^0(\bar\Omega)} \leq C \quad \forall t \in [t_0, t_0+T_{\max}).
\end{equation}
Therefore, $T_{\max}=\infty$.}

First, by the variation of constant formula and the first equation in \eqref{u-v-eq00}, we have
\vspace{-0.1in}\begin{align*}
u(t)&=e^{{ -(t-t_0)A}}u_0-\chi\int_{t_0}^t e^{{ -(t-s)A}}\nabla(u (s)\cdot\nabla v(s))ds \\
&\,\,+\int_{t_0}^t e^{{ -(t-s)A}}u(s)\Big[\underbrace{1+a_0(s,\cdot)-a_1(s,\cdot)u(s)-(a_2(s,\cdot))_+\int_{\Omega}u(s) +(a_2(s,\cdot))_-\int_{\Omega}u(s)}_{I_0(\cdot,s)}\Big] ds,
\end{align*}
where $A=-\Delta+I$.
Note that
$$u(s)I_0(\cdot,s)\leq u(s)[\underbrace{1+(a_{2,\inf})_-M_0+a_0(\cdot,s)-a_1(s,\cdot)u(s)}_{I_1(\cdot,s)}].
 $$
By the comparison principle for parabolic equations, we get
$$\int_{t_0}^t e^{{ -(t-s)A}}u(s)I_0(\cdot,s)ds\leq \int_{t_0}^t e^{{-(t-s)A}}u(s)I_1(\cdot,s)ds.$$
 Therefore
$$u(t)\leq  u_1(t)+u_2(t)+u_3(t),$$
where
$$u_1(t)=e^{{ -(t-t_0)A}}u_0,\quad u_2(t)=-\chi\int_{t_0}^t e^{{ -(t-s)A}}\nabla(u (s)\cdot\nabla v(s))ds$$
 and
 $$u_3(t,x)=\int_{t_0}^t e^{{ -(t-s)A}}u(s)\left[1+(a_{2,\inf})_-M_0+a_0(\cdot,s)-a_1(s,\cdot)u(s) \right] ds.$$

 Next, note that there are $c_0,c_1>0$ such that $(1+(a_{2,\inf})_-M_0+a_0(t,x))r-a_1(t,x)r^2\le c_0-c_1 r^2$ for all $t\in\RR$, $x\in\Omega$, and $r\ge 0$.
We then  have that
\begin{equation}
\label{aux-new-eq3}
  \| u_1(t)\|_{L^\infty(\Omega)} \leq e^{-(t-t_0)} \|u_0\|_{L^{\infty}(\Omega)} \quad \forall \,\,  t\in [t_0,t_0+T_{\max})
  \end{equation}
and
\begin{equation}
\label{aux-new-eq4}
u_3(t) \leq  C \int_{t_0}^t  e^{{-(t-s)A}} ds \leq C \int_{t_0}^t  e^{-(t-s) }\leq C\quad \forall\,\, t\in [t_0,t_0+T_{\max}).
\end{equation}

Choose $p>n$ and $\alpha \in (\frac{n}{2p},\frac{1}{2}).$ Then $X^{\alpha} \subset L^{\infty}(\Omega)$ and the inclusion is continuous (see \cite{DH77} exercise 10, page 40.) Choose $\epsilon \in (0, \frac{1}{2}-\alpha)$.
By Lemma \ref{estimates-lm-00}(ii) and (iii), we have
\vspace{-0.1in}\begin{eqnarray*}
\|u_2(t)\|_{L^{\infty}(\Omega)}&\leq & C\|A^{\alpha}u_2(t)\|_{L^p(\Omega)} \nonumber\\
&\le & C|\chi| \int_{t_0}^t \| A^{\alpha}e^{{-\frac{t-s}{2}A}} e^{{ -\frac{t-s}{2}A}}\nabla(u(s) \cdot\nabla v(s))\|_{L^p(\Omega)}ds \nonumber \\
  &\leq& C|\chi| K_2(p,\alpha) K_3(p)\int_{t_0}^t\Big(1+(t-s)^{-\alpha-\frac{1}{2}}\Big)e^{-\frac{t-s}{2}} \| u(s) \cdot\nabla v(s)\||_{L^p(\Omega)}ds \nonumber\\
 &\leq&  C|\chi| K_2(p,\alpha)K_3(p)\int_{t_0}^t\Big(1+(t-s)^{-\alpha-\frac{1}{2}}\Big)e^{-\frac{t-s}{2}} \| u(s)\| _{L^{p_1}(\Omega)}  \|\nabla v(s)\|_{L^{p_2}(\Omega)}ds \nonumber
\end{eqnarray*}
for $t\in[t_0,t_0+T_{\max})$,
where $p_1>p $ and $\frac{1}{p}=\frac{1}{p_1}+\frac{1}{p_2}.$ By \eqref{new-global-eq3} and \eqref{new-global-step4-eq00}, we get
\begin{align}
\label{aux-new-eq5}
&\|u_2(t)\|_{L^{\infty}(\Omega)}\nonumber\\
 &\leq C(\sup_{t\in[t_0,t_0+T_{\max})}\|u(t)\|_{L^{p_1}(\Omega)},\sup_{t\in[t_0,t_0+T_{\max})}\|v(t)\|_{L^{p_2}(\Omega)})\int_{t_0}^{\infty} (t-s)^{-\alpha-\frac{1}{2}}e^{-\frac{t-s}{2} }ds.
\end{align}

Now from the second equation in \eqref{u-v-eq00} and the comparison principle for parabolic equations, we get
$$
\|v(t)\|_{C^0(\bar{\Omega})}\leq \max\{\|v_0\|_\infty, \sup_{t_0\le t<t_0+T_{\max}}\frac{\mu}{\lambda}\|u(t)\|_{C^0(\bar{\Omega})}\}
$$
\eqref{new-global-eq4} then   follows. This implies that $T_{\max}=\infty$.

\medskip

\noindent{\bf Step 6.} {\it In this step, we prove that \eqref{m1-eq} and \eqref{m2-eq} hold.}

\medskip

First, \eqref{m1-eq} follows from \eqref{L1-bound-eq00}. We then only need to prove \eqref{m2-eq}.

By the arguments in Step 1, we have for any $\gamma \in  \Big(1,\frac{|\chi |C^{\frac{1}{\gamma+1}}_{\gamma+1}{ \mu^\frac{1}{\gamma+1}}}{\big(|\chi | C^{\frac{1}{\gamma+1}}_{\gamma+1}-a_{1,\inf}\big)_+} \Big)$ and $t>t_0+t^1$ that
{
\begin{align*}
&\frac{1}{\gamma}\int_\Omega u^\gamma\nonumber\\
  & \leq\Big(A_\gamma r_0^{-\gamma}C_{\gamma+1}[\|v(\cdot,t_0+t^1)\|^{\gamma+1}_{L^{\gamma+1}}+\|\nabla v(\cdot,t_0+t^1)\|^{\gamma+1}_{L^{\gamma+1}}]+\frac{1}{\gamma}\|u(\cdot,t_0+t^1)\|_\infty\Big)e^{-(\gamma+1)(t-t_0-t^1)}\\
&\,\,\,\, +\underbrace{\frac{1}{\gamma+1}\big[\frac{\gamma+1}{\gamma}\epsilon\big]^{-\gamma}[a_{0,\sup}+\frac{\gamma+1}{\gamma}+(a_{2,\inf})_- M_1\Big]^{\gamma+1}|\Omega|}_{C_1(\epsilon,a_i,\gamma):=C_1},
\end{align*}
}
where $r_0=\big(\gamma A_\gamma C_{\gamma+1}\mu\big)^{\frac{1}{\gamma+1}}|\chi|$  (see \eqref{proof-global-2-eq09}).
Therefore,  there is $ t_{1}>t^1$ such for any $\gamma \in  \Big(1,\frac{|\chi |C^{\frac{1}{\gamma+1}}_{\gamma+1}{ \mu^\frac{1}{\gamma+1}}}{\big(|\chi | C^{\frac{1}{\gamma+1}}_{\gamma+1}-a_{1,\inf}\big)_+} \Big)$,
\begin{equation}
\label{largetime-bound-eq00}
\int_\Omega u^\gamma \leq \gamma[1+C_1]:=C_1(a_i,\gamma)\,\, \forall\, t>t_0+ t_1.
\end{equation}

Next, by the arguments in Step 2 (in particular, by  \eqref{new-global-step2-eq03b}),  there is $ t_{2}> t_{1}$ such for any $q \in \left[1,\frac{nq_0}{(n-q_0)_+}\right)$,
\begin{equation}
\label{largetime-bound-eq01}
  \|\nabla v(t)\|_q \leq 1 +C_2(a_i,\gamma,q)\,\, \forall\, t>t_0+ t_{2}.
\end{equation}

Now, by \eqref{largetime-bound-eq00}, \eqref{largetime-bound-eq01}, and the arguments in Step 3 (in particular,
\eqref{proof-global-step3-eq03}), there exists $ t_{3}> t_{2}$ such that any $\gamma>1,$  we have
\begin{equation}
\label{largetime-bound-eq03}
  \int_\Omega u^\gamma  \leq C_3(a_i,\gamma)\,\, \forall \, t>t_0+ t_{3}.
\end{equation}

 Finally, by  \eqref{largetime-bound-eq03} and the arguments in Step 4
 (in particular, \eqref{aux-new-eq2}), there exists $\exists t_{4}> t_{3}$ such that any $q>1,$  we get
\begin{equation}
\label{largetime-bound-eq04}
  \|\nabla v(t)\|_q \leq C_4(a_i,\gamma,q),\, \forall t>t_0+ t_{4}.
\end{equation}
\eqref{m2-eq} then follows from \eqref{largetime-bound-eq03} , \eqref{largetime-bound-eq04} and the proof of Step 5  (in particular, \eqref{aux-new-eq3}, \eqref{aux-new-eq4}, and \eqref{aux-new-eq5}).
\end{proof}

We now prove Theorem \ref{thm-global-000} under the assumption {\bf (H2)}.

\begin{proof}[Proof of Theorem \ref{thm-global-000} with the assumption {\bf (H2)}]
{ Assume that {\bf (H2)} holds.   Theorem \ref{thm-global-000}  can be proved by properly modifying the arguments in \cite[Lemma 3.1]{Win2014}. For the completeness, we also provide a proof in the following.}

First, we have
\begin{align*}
  \frac{1}{2}\frac{d}{dt}|\nabla v|^2=  \frac{1}{2}\big[\sum_{i=1}^{n}((v_{x_i})^2)_t\big] =  \sum_{i=1}^{n} v_{x_i} (v_{x_i})_t=  \sum_{i=1}^{n} v_{x_i} ( v_t)_{x_i}.
\end{align*}
From the second equation of \eqref{u-v-eq00}, we get
\begin{align*}
  \frac{1}{2}\frac{d}{dt}|\nabla v|^2 =& \sum_{i=1}^{n} v_{x_i} (\Delta v - \lambda v+ \mu u)_{x_i} \\
  = & \sum_{i=1}^{n} v_{x_i} (\Delta v_{x_i} - \lambda v_{x_i}+ \mu u_{x_i}) \\
  = &  \nabla v \cdot \nabla(\Delta v)-\lambda  |\nabla v|^2+\mu  \nabla v \cdot \nabla u.
\end{align*}
Combining this with $\nabla v \cdot \nabla(\Delta v)=\frac{1}{2}\Delta |\nabla v|^2-|D^2v|^2,$ we get
\begin{equation}\label{proof-global-eq00}
 {  \frac{1}{2\mu }\frac{d}{dt}|\nabla v|^2=\frac{1}{2\mu}\Delta |\nabla v|^2-\frac{1}{\mu}|D^2v|^2-\frac{ \lambda}{\mu} |\nabla v|^2+ \nabla v \cdot \nabla u.}
\end{equation}

Next, by multiplying the first equation of \eqref{u-v-eq00} by {  $\frac{1}{{ |\chi|}},$
 we get
\begin{equation}\label{proof-global-eq01}
\frac{1}{{|\chi|}}u_t=\frac{1}{{ |\chi|}}\Delta u-  \nabla u \cdot \nabla v-  u \Delta v+\frac{1 }{{ |\chi|}}u\Big(a_0(t,x)-a_1(t,x)u-a_2(t,x)\int_{\Omega}u\Big).
\end{equation}
By adding \eqref{proof-global-eq00} and \eqref{proof-global-eq01}, we get
\begin{align}\label{proof-global-eq02}
 \frac{d}{dt}\big[\frac{1 }{{ |\chi|}}u+\frac{1}{2\mu}|\nabla v|^2\big]  =& \Delta \big[\frac{1}{{ |\chi|}}u+\frac{1}{2\mu}|\nabla v|^2\big]-\frac{1}{\mu}|D^2v|^2-\frac{\lambda}{\mu}  |\nabla v|^2-  u \Delta v \nonumber\\
  +&\frac{1}{{|\chi|}}u\Big(a_0(t,x)-a_1(t,x)u-a_2(t,x)\int_{\Omega}u\Big).
\end{align}
By Young's inequality, we have $$| u \Delta v|\leq \frac{n\mu}{4}u^2+\frac{1}{\mu}|D^2v|^2.$$ By combining this  with \eqref{proof-global-eq02}, we get
\begin{align}\label{proof-global-eq03}
  \frac{d}{dt}\big[\frac{1}{{ |\chi|}}u+\frac{1}{2\mu}|\nabla v|^2\big] \leq  & \Delta \big[\frac{1}{{ |\chi|}}u+\frac{1}{2\mu}|\nabla v|^2\big]-\frac{\lambda}{\mu} |\nabla v|^2+\frac{n\mu}{4}u^2 \nonumber \\
  & +  \frac{1}{{ |\chi|}}u\Big(a_0(t,x)-a_1(t,x)u-a_2(t,x)\int_{\Omega}u\Big) \nonumber\\
  \leq & \Delta \big[\frac{1}{{ |\chi|}}u+\frac{1}{2\mu}|\nabla v|^2\big]-\frac{\lambda}{\mu} |\nabla v|^2-\frac{1}{{ |\chi|}}\big(a_{1,\inf}-\frac{ n\mu{ |\chi|}}{4}\big) u^2-\frac{2\lambda}{{ |\chi|}}u \nonumber\\
 & +   \frac{1}{{ |\chi|}}u\Big(a_{0,\sup}+2\lambda +\sup_{t \in \mathbb{R}}(a_{2,\inf}(t))_-\int_{\Omega}u\Big).
\end{align}
}
{Let $M_{0,a_i,\|u_0\|_\infty}=a_{0,\sup}+2\lambda +\sup_{t \in \mathbb{R}}(a_{2,\inf}(t))_-M_0(\|u_0\|_\infty).$  Then, by \eqref{L1-bound-eq000}, \eqref{proof-global-eq03} becomes for $t_0<t<t_0+T_{\max},$
\begin{align}\label{proof-global-eq04}
  \frac{d}{dt}\big[\frac{1}{{ |\chi|}}u+\frac{1}{2\mu}|\nabla v|^2\big] \leq &\Delta \big[\frac{1}{{ |\chi|}}u+\frac{1}{2\mu}|\nabla v|^2\big]-\frac{\lambda}{\mu} |\nabla v|^2-\frac{1}{{ |\chi|}}\big(a_{1,\inf}-\frac{ n\mu{ |\chi|}}{4\mu}\big) u^2\nonumber\\
  &-\frac{2\lambda}{{|\chi|}}u +\frac{1}{{ |\chi|}} M_{0,a_i,\|u_0\|_\infty}u \nonumber\\
  \leq  & \Delta \big[\frac{1}{{ |\chi|}}u+\frac{1}{2\mu}|\nabla v|^2\big]-2\lambda \big [\frac{1}{{ |\chi|}}u+\frac{1}{2\mu}|\nabla v|^2\big]\nonumber\\
   & -\frac{1}{{ |\chi|}}\big(a_{1,\inf}-\frac{ n\mu {|\chi|}}{4}\big)\big[u^2-\frac{ M_{0,a_i,\|u_0\|_\infty}}{a_{1,\inf}-\frac{ n\mu { |\chi|}}{4}}u\big]\nonumber\\
  = & \Delta \big[\frac{1}{{ |\chi|}}u+\frac{1}{2\mu}|\nabla v|^2\big]-2\lambda \big [\frac{1}{{ |\chi|}}u+\frac{1}{2\mu}|\nabla v|^2\big]\nonumber\\
  &-\frac{1}{{ |\chi|}}\big(a_{1,\inf}-\frac{ n { \mu|\chi|}}{4}\big)\big(u-\frac{ M_{0,a_i,\|u_0\|_\infty}}{2(a_{1,\inf}-\frac{ n\mu { |\chi|}}{4})}\big)^2\nonumber\\
  &+\frac{1}{{|\chi|}}\big(a_{1,\inf}-\frac{n { \mu |\chi|}}{4}\big)\frac{ M_{0,a_i,\|u_0\|_\infty}^2}{4(a_{1,\inf}-\frac{\mu n { |\chi|}}{4})^2}.
\end{align}
Thus since $\big(a_{1,\inf}-\frac{n\mu { |\chi|}}{4}\big)>0$, we get for $t_0<t<t_0+T_{\max},$
\begin{equation}\label{proof-global-eq05}
  \frac{d}{dt}\big[\frac{1}{{ |\chi|}}u+\frac{1}{2\mu}|\nabla v|^2\big] \leq \Delta \big[\frac{1}{{|\chi|}}u+\frac{1}{2\mu}|\nabla v|^2\big]-2\lambda \big [\frac{1}{{ |\chi|}}u+\frac{1}{2\mu}|\nabla v|^2\big]+\frac{ M_{0,a_i,\|u_0\|_\infty}^2}{4{ |\chi|}(a_{1,\inf}-\frac{n\mu { |\chi|}}{4})}.
\end{equation}
 Therefore since $\frac{\p v}{\p n}= 0$ and  $\Omega$ is convex, it follows from \cite[Lemma 3.2]{TaoWin12} that $\frac{\p |\nabla v|^2}{\p n}\leq 0.$ Thus $z=\frac{1}{{|\chi|}}u+\frac{1}{2\mu}|\nabla v|^2$ solve
\begin{equation}\label{proof-global-eq04}
  \begin{cases}
z_t \leq  \Delta z-z+\frac{ M_{0,a_i,\|u_0\|_\infty}^2}{4{ |\chi|}(a_{1,\inf}-\frac{n \nu{|\chi|}}{4})}\cr
\frac{\p z}{\p n}\leq 0.
\end{cases}
\end{equation}
By the comparison principle for parabolic equations, we get
\[0\leq z(\cdot,t)\leq \max\{z(\cdot,t_0),\frac{ M_{0,a_i,\|u_0\|_\infty}^2}{4{ |\chi|}(a_{1,\inf}-\frac{n\mu { |\chi|}}{4})}\}\,\,\,\, \forall \, t_0\leq t<t_0+T_{\max}.\]
Therefore, it follows by the blow-up criterion \eqref{local-existence-eq00} that $T_{max}=\infty$,
and \eqref{m2-eq} follows from the above arguments.}
\end{proof}

\section{Pointwise persistence}
\label{S:Persistence}
{  In this section, we investigate the pointwise persistence in \eqref{u-v-eq00} and prove Theorem \ref{thm-persistence-entire solution-000}.

 Throughout this section, we assume that {\bf (H1)} or {\bf (H2)} holds, and that $t^1(u_0,v_0)$, $t^2(u_0,v_0)$, $M_1$, and $M_2$ are
 as in Theorem \ref{thm-global-000}.
 We start by proving the following three important Lemmas.

\begin{lemma}
\label{persisetence-full-lemma-00} Let $p>1$ be given. There is $C_1(p)>0$ such that
for any  $t_0 \in \mathbb{R}$, $(u_0,v_0) \in C(\bar{\Omega})\times W^{1,\infty}(\Omega)$ with $u_0,v_0 \geq 0,$  and any
$\tilde t_0> t_0$,
there holds
\begin{align}
\label{new-new-eq1}
\|\nabla v(t)\|_{L^p(\Omega)}\le  C_1(p)\Big(\frac{(t-\tilde t_0)}{\tau}\Big)^{-\frac{1}{2}}e^{-\frac{\lambda(t-\tilde t_0)}{\tau}}\|v(\cdot,\tilde t_0)\|_{L^\infty(\Omega)} +  C_1(p)\sup_{\tilde t_0\le s\le t}\|u(\cdot,s)\|_\infty\quad \forall\, t\ge \tilde t_0.
\end{align}
\end{lemma}

\begin{proof}
 By the second equation in \eqref{u-v-eq00} and the variation of constant formula, we have for  $t> {\tilde t_0}$ that
$$v(\cdot,t)=e^{\frac{1}{\tau}{ (t-\tilde t_0)(\Delta-\lambda I)}}v(\cdot,\tilde t_0)+\frac{\mu}{\tau}\int_{\tilde t_0}^te^{\frac{1}{\tau}{ (t-s)(\Delta-\lambda I)}}u(\cdot,s)ds.$$
Thus for $p>1,$ we have
\begin{equation}
\label{aaux-new-eq0}
\|\nabla v\|_{L^p(\Omega)}\leq \underbrace{\|\nabla e^{\frac{1}{\tau}{(t-\tilde t_0)(\Delta-\lambda I)}}v(\cdot,\tilde t_0)\|_{L^p(\Omega)}}_{I_1}+\underbrace{\frac{\mu}{\tau}\int_{\tilde t_0}^t \|\nabla e^{\frac{1}{\tau}{(t-s)(\Delta- \lambda I)}}u(\cdot,s)\|_{L^p(\Omega)}ds }_{I_2}.
\end{equation}
Then by Lemma \ref{estimates-lm-00} (vi),
\begin{align}
\label{aaux-new-eq1}
I_1&=e^{-\frac{\lambda (t-\tilde t_0)}{\tau}}\|\nabla e^{{(\frac{t-\tilde t_0}{\tau})}\Delta}v(\cdot,\tilde t_0)\|_{L^p(\Omega)}\nonumber\\
&\leq K_6(p)( \frac{t-\tilde t_0}{\tau})^{-\frac{1}{2}}e^{-\frac{\lambda(t-\tilde t_0)}{\tau}}\|v(\cdot,\tilde t_0)\|_{L^\infty(\Omega)}\quad \forall\, t>\tilde t_0.
\end{align}
By Lemma \ref{estimates-lm-00}(vi) again, we have for $t>\tilde t_0$ that
\begin{align}
\label{aaux-new-eq2}
I_2 &\leq  \frac{\mu}{\tau}K_6(p)\sup_{\tilde  t_0\le s\le t}\|u(\cdot,s)\|_\infty \int_{\tilde t_0}^t (\frac{t-s}{\tau})^{-\frac{1}{2}}e^{-\lambda (\frac{t-s}{\tau})}ds \nonumber\\
      & \leq   \frac{\mu}{\tau} K_6(p)\sup_{\tilde t_0\le s\le t}\|u(\cdot,s)\|_\infty \tau \int_{0}^\infty s^{-\frac{1}{2}}e^{-\lambda s}ds.
\end{align}
The lemma then follows from
 \eqref{aaux-new-eq0}-\eqref{aaux-new-eq2}.
\end{proof}

\begin{corollary}
\label{cor1}
There is $\tilde C_1(p)$ such that for any for any  $t_0 \in \mathbb{R}$, and $(u_0,v_0) \in C(\bar{\Omega})\times W^{1,\infty}(\Omega)$ with $u_0,v_0 \geq 0,$  there is $t^3=t^3(u_0,v_0)>t^2(u_0,v_0)$ satisfying that
\begin{equation}
\label{new-new-eq1-1}
\|\nabla v(t)\|_{L^p}\le\tilde C_1(p) \quad \forall\, t\ge t_0+t^3.
\end{equation}
\end{corollary}

\begin{proof}
Choose $\tilde t_0={ t_0}+t^2(u_0,v_0)$ in Lemma \ref{persisetence-full-lemma-00}. By Theorem \ref{thm-global-000},
$$
\|u(\cdot,s)\|_\infty\le M_2\quad \forall\,\, s\ge \tilde t_0.
$$
This together with \eqref{new-new-eq1} implies that there is $t^3(u_0,v_0)$ such that
$$
\|\nabla v\|_{L^p(\Omega)}\le C_1(p)\big(1+M_2) \quad \forall\, t\ge t_0+t^3.
$$
The corollary then follows with $\tilde C_1(p)=C_1(p)\big(1+M_2\big)$.
\end{proof}

\begin{lemma}
\label{persisetence-full-lemma-01}
 Fix $0<\eta<\frac{1}{2}$ and $p>1$. Let $A=-\Delta +\alpha I$ for some $\alpha\in (0,1)$ with $D(A)=\{u \in  W^{2,p}(\Omega) : \frac{\p u}{\p n}=0 \}$.
  There is $C_2(p,\eta)>0$ such that  for any  $t_0 \in \mathbb{R}$, $(u_0,v_0) \in C(\bar{\Omega})\times W^{1,\infty}(\Omega)$ with $u_0,v_0 \geq 0,$ and any $\tilde t_0>t_0$, there holds
 \begin{align}
 \label{new-new-eq2}
& \|A^\eta u(\cdot,t)\|_{L^p(\Omega)}\nonumber\\
&\le C_2(p,\eta) (t- \tilde t_0)^{-\eta}e^{-(1-\alpha )(t-\tilde t_0)}\|u(\cdot,\tilde t_0)\|_p\nonumber \\
 &\,\,\,\,
 +C_2(p,\eta)\sup_{\tilde t_0\le s\le t}\|u(\cdot,s)\|_\infty\big(1+\sup_{\tilde t_0\le s\le t}\|\nabla v(\cdot,s)\|_{L^p(\Omega)}+\sup_{\tilde t_0\le s\le t}\|u(\cdot,s)\|_\infty\big)
 \end{align}
 for all $t{ > }\tilde t_0$.
\end{lemma}

\begin{proof}
    By the first equation in \eqref{u-v-eq00} and  the variation of constant formula, we have for $t> \tilde t_0$ that
\begin{align*}
u(\cdot,t)&=\underbrace{e^{{(t- \tilde t_0)}(\Delta-I)}u(\cdot,\tilde t_0)}_{I_1}-\chi \underbrace{\int_{\tilde t_0}^t e^{{(t-s)(\Delta-I)}}\nabla\cdot (u(\cdot,s) \nabla v(\cdot,s))ds}_{I_2} \nonumber \\
&+\underbrace{\int_{\tilde t_0}^t e^{{(t-s)(\Delta-I)}}u(\cdot,s)\Big(a_0(s,x)+1-a_1(t,x)u(\cdot,s)-a_2(s,x)\int_{\Omega}u(\cdot,s)\Big)ds}_{I_3}.
\end{align*}
Thus
\begin{equation}
\label{step2-eq00}
\|A^\eta u(\cdot,t)\|_p \leq \|A^\eta I_1\|_p+\chi\|A^\eta I_2\|_p  +\|A^\eta I_3\|_p.
\end{equation}

 We first estimate $\|A^\eta I_1\|_p$.
Note that
\begin{align*}
 \|A^\eta I_1\|_p&=\|A^\eta e^{(\Delta-I)(t- \tilde t_0)}u(\cdot,\tilde t_0)\|_p\\
                &=\|A^\eta e^{(\Delta-\alpha I)(t- \tilde t_0)}e^{-(1-\alpha)(t- \tilde t_0)}u(\cdot,\tilde t_0)\|_p \\
&=e^{-(1-\alpha )(t- \tilde t_0)}\|A^\eta e^{(\Delta-\alpha I)(t- \tilde t_0)}u(\cdot,\tilde t_0)\|_p.
\end{align*}
Then by  Lemma \ref{estimates-lm-00}(ii),
\begin{equation}
\label{step2-eq01}
\|A^\eta I_1\|_p\leq  K_2(p,\eta) (t- \tilde t_0)^{-\eta}e^{-(1-\alpha )(t- \tilde t_0)}\|u(\cdot,\tilde t_0)\|_p.
\end{equation}

Next, we estimate $\|A^\eta I_2\|_p$. Note that
\begin{align*}
 \|A^\eta I_2\|_p& \leq \int_{\tilde t_0}^t  \|A^\eta e^{{ (t-s)(\Delta-I)}}\nabla\cdot (u(\cdot,s) \nabla v(\cdot,s))\|_pds\\
                &= \int_{\tilde t_0}^t  \|A^\eta  e^{(\Delta-\alpha I)\frac{(t-s)}{2}} \big(e^{-(1-\frac{\alpha}{2} )(t- s)}e^{{\frac{(t-s)}{2}}\Delta}\nabla\cdot (u(\cdot,s) \nabla v(\cdot,s))\big)\|_pds.
\end{align*}
Thus by Lemma \ref{estimates-lm-00}(ii) again, we have
$$
 \|A^\eta I_2\|_p  \leq K_2(p,\eta)  \int_{\tilde t_0}^t \big(\frac{t-s}{2}\big)^{-\eta} e^{-(1-\frac{\alpha}{2} )(t- s)} \|e^{{ \frac{(t-s)}{2}}\Delta}\nabla\cdot (u(\cdot,s) \nabla v(\cdot,s))\big)\|_pds.
$$
By  Lemma \ref{estimates-lm-00} (iii), we have
\begin{align}
\label{step2-eq02}
 &\|A^\eta I_2\|_p\nonumber\\
  & \leq { K_{2,3}(p,\eta)}   \int_{\tilde t_0}^t \big(\frac{t-s}{2}\big)^{-\eta}\big(1+ \big(\frac{t-s}{2}\big)^{-\frac{1}{2}} \big) e^{-(1-\frac{\alpha}{2} )(t- s)} \| u(\cdot,s) \nabla v(\cdot,s))\|_pds\nonumber\\
 & \leq K_{2,3}(p,\eta)   \int_{\tilde t_0}^t \big(\frac{t-s}{2}\big)^{-\eta}\big(1+ \big(\frac{t-s}{2}\big)^{-\frac{1}{2}} \big) e^{-(1-\frac{\alpha}{2} )(t- s)} \| u(\cdot,s)\|_\infty \| \nabla v(\cdot,s))\|_pds\nonumber \\
 & \le  K_{2,3}(p,\eta)  { \sup_{\tilde t_0\le s\le t}\big(\|u(\cdot,s)\|_\infty \|\nabla v(\cdot,s)\|_p\big)} \int_{\tilde t_0}^t \big(\frac{t-s}{2}\big)^{-\eta}\big(1+ \big(\frac{t-s}{2}\big)^{-\frac{1}{2}} \big) e^{-(1-\frac{\alpha}{2} )(t- s)}ds,
\end{align}
{ where $K_{2,3}(p,\eta)=K_2(p,\eta) K_3(p)$}. 
Note the last integral in \eqref{step2-eq02} is finite because $\eta<\frac{1}{2}.$

Third, we have
\begin{align*}
&\|A^\eta I_3\|_p\\
 &\leq \int_{\tilde t_0}^t  \|A^\eta e^{{ (t-s)(\Delta-I)}}u(\cdot,s)\Big(a_0(s,\cdot)+1-a_1(t,\cdot)u(\cdot,s)-a_2(s,\cdot)\int_{\Omega}u(\cdot,s)\Big)\|_pds\\
&=\int_{\tilde t_0}^t  \|A^\eta e^{{(t-s)(\Delta-\alpha I)}} e^{-(1-\alpha )(t- s)}u(\cdot,s)\Big(a_0(s,\cdot)+1-a_1(t,\cdot)u(\cdot,s)-a_2(s,\cdot)\int_{\Omega}u(\cdot,s)\Big)\|_pds.
\end{align*}
By Lemma \ref{estimates-lm-00}(ii), we have that
\begin{align}
\label{step2-eq03}
&\|A^\eta I_3\|_p \nonumber\\
&\leq  K_2(p,\eta)|\Omega|^\frac{1}{p}\sup_{\tilde t_0\le s\le t}\|u(\cdot,s)\|_\infty \Big(a_{0,\sup}+1+[\sup|a_2||\Omega|+a_{1,\sup}]
\sup_{\tilde t_0\le s\le t}\|u(\cdot,s)\|_\infty \Big)\nonumber\\
&\,\, \cdot \int_0^\infty  s^{-\eta} e^{-(1-\alpha )s}ds.
\end{align}
\eqref{new-new-eq2} then follows from \eqref{step2-eq01}-\eqref{step2-eq03}.
\end{proof}

\begin{corollary}
\label{cor2}
There is $\tilde C_2(p,\eta)$ such that for any  $t_0 \in \mathbb{R}$, and $(u_0,v_0) \in C(\bar{\Omega})\times W^{1,\infty}(\Omega)$ with $u_0,v_0 \geq 0,$ there is $t^4=t^4(u_0,v_0)\ge t^3(u_0,v_0)$ satisfying that
\begin{equation}
\label{new-new-eq2-1}
\|A^\eta u(\cdot,t)\|_p\le \tilde C_2(p,\eta)\quad \forall\, t\ge t_0+ t^4.
\end{equation}
\end{corollary}

\begin{proof}
It follows from \eqref{m2-eq}, \eqref{new-new-eq1-1}, and \eqref{new-new-eq2}.
\end{proof}

\begin{lemma}
\label{persisetence-full-lemma-03}
Fix $0<\eta<\frac{1}{2}$ and $0<\alpha<\min\{1,\lambda\}$.  Choose $\theta>0$ and $p>n$ such that $2\theta-\frac{n}{p}>2,$ and $\theta<1+\eta$.
 Let $A=-\Delta +\alpha I$ with $D(A)=\{u \in  W^{2,p}(\Omega) : \frac{\p u}{\p n}=0 \}$.
There is $C_3(p,\theta,\eta)$ such that for any  $t_0 \in \mathbb{R}$, $(u_0,v_0) \in C(\bar{\Omega})\times W^{1,\infty}(\Omega)$ with $u_0,v_0 \geq 0,$  and any $\tilde t_0>t_0$, there holds
\begin{align}
\label{new-new-eq3}
\|v(\cdot,t)\|_{W^{2,\infty}(\Omega)}&\le  C_3(p,\theta,\eta)(\frac{t-\tilde t_0}{\tau})^{-\theta}e^{-(\lambda-\alpha)(\frac{t-\tilde t_0}{\tau})}\|v(\cdot,\tilde t_0)\|_{L^\infty(\Omega)}\nonumber\\
&\,\, +C_3(p,\theta,\eta)\sup_{\tilde t_0\le s\le t} \|A^\eta u(\cdot,s)\|_{L^p(\Omega)}\quad \forall\, t\ge \tilde t_0.
\end{align}
\end{lemma}

\begin{proof}
   Note that, for $t> \tilde t_0$, we have
$$v(\cdot,t)=e^{\frac{1}{\tau}{ (t-\tilde t_0)(\Delta-\lambda I)}}v(\cdot,\tilde t_0)+\frac{\mu}{\tau}\int_{\tilde t_0}^te^{\frac{1}{\tau}{(t-s)(\Delta-\lambda I)}}u(\cdot,s)ds.$$
This implies that there is $\tilde C>0$ such that
\begin{align}
\label{step3-eq01}
&\|v(\cdot,t)\|_{W^{2,\infty}(\Omega)}\nonumber\\
 &\leq \tilde C\|A^\theta v(\cdot,t)\|_p \nonumber\\
&\leq \tilde Ce^{-(\lambda-\alpha)(\frac{t-\tilde t_0}{\tau})}\|A^\theta e^{{ -(\frac{t-\tilde t_0}{\tau})A}}v(\cdot,\tilde t_0)\|_{L^p(\Omega)}\nonumber\\
&\,\,\,\, +\tilde C \frac{\mu}{\tau}\int_{\tilde t_0}^t \|A^{(\theta-\eta)} e^{{-(\frac{t-s}{\tau})A}}A^\eta u(\cdot,s)\|_{L^p(\Omega)}e^{-(\lambda-\alpha)(\frac{t-s}{\tau})}ds \nonumber\\
&\leq \tilde CK_2(p,\theta)(\frac{t-\tilde t_0}{\tau})^{-\theta}e^{-(\lambda-\alpha)(\frac{t-\tilde t_0}{\tau})}\|v(\cdot,\tilde t_0)\|_{L^p(\Omega)}\nonumber\\
&\,\, \,\,+\tilde C\frac{\mu}{\tau}K_2(p,\theta-\eta)\int_{\tilde t_0}^t (\frac{t-s}{\tau})^{-(\theta-\eta)}e^{-(\lambda-\alpha)(\frac{t-s}{\tau})}\|A^\eta u(\cdot,s)\|_{L^p(\Omega)}ds \nonumber\\
&\leq \tilde CK_2(p,\theta)(\frac{t-\tilde t_0}{\tau})^{-\theta}e^{-(\lambda-\alpha)(\frac{t-\tilde t_0}{\tau})}\|v(\cdot,\tilde t_0)\|_{L^p(\Omega)}\nonumber\\
&\,\, \,\, +\tilde C\frac{\mu}{\tau}K_2(p,\theta-\eta)\sup_{\tilde t_0\le s\le t} \|A^\eta u(\cdot,s)\|_{L^p(\Omega)}\int_{\tilde t_0}^t (\frac{t-s}{\tau})^{-(\theta-\eta)}e^{-(\lambda-\alpha)(\frac{t-s}{\tau})}ds.
\end{align}
The lemma then follows.
\end{proof}

\begin{corollary}
\label{cor3}
There is $\tilde C_3$ such that
for any  $t_0 \in \mathbb{R}$, and $(u_0,v_0) \in C(\bar{\Omega})\times W^{1,\infty}(\Omega)$ with $u_0,v_0 \geq 0,$  there is $t^5=t^5(u_0,v_0)\ge t^4(u_0,v_0)$ satisfying that
\begin{equation}
\label{new-new-eq3-1}
\|v(\cdot,t)\|_{W^{2,\infty}(\Omega)}\le\tilde C_3\quad \forall \, t\ge t_0+t^5.
\end{equation}
\end{corollary}

\begin{proof}
It follows from \eqref{new-new-eq2-1} and \eqref{new-new-eq3}.
\end{proof}

We now prove Theorem  \ref{thm-persistence-entire solution-000}.

\begin{proof}[Proof of Theorem \ref{thm-persistence-entire solution-000}]
We divide the proof into five steps. For given $t_0\in\RR$,  $(u_0,v_0) \in C(\bar{\Omega})\times W^{1,\infty}(\Omega)$ with $u_0,v_0 \geq 0,$  and $\epsilon>0$, let $t^5(u_0,v_0)$ be as in Corollary \ref{cor3}. Let
\begin{equation}
\label{c-start-eq}
C^*=C_3(p,\theta,\eta) \big(1+3C_2(p,\eta) + 2 C_1(p) C_2(p,\eta)\big)
\end{equation}
 and
 \begin{equation}
 \label{epsilon-star-eq}
 \epsilon^*:=\frac{a_{0,\inf}}{a_{1,\sup}+\chi C^*+|\Omega|(a_{2,\sup})_+}.
 \end{equation}
Fix $0<\epsilon_0<\min\{\epsilon^*,1\}$. Let $T^*>0$ be such that
$$
\max\Big\{(\frac{T^*}{\tau})^{-\frac{1}{2}}e^{-(\frac{T^*}{\tau})}\tilde C_3,  (T^*)^{-\eta}e^{-(1-\alpha )(T^*)}M_2,
(\frac{T^*}{\tau})^{-\theta}e^{-(\lambda-\alpha)(\frac{T^*}{\tau})}\tilde C_3\Big\}<\epsilon_0.
$$

\medskip

\noindent{\bf Step 1.} {\it   In this step, we prove that for any $t_0\in\RR$ and  $(u_0,v_0) \in C(\bar{\Omega})\times W^{1,\infty}(\Omega)$ with $u_0,v_0 \geq 0$ and $u_0\not \equiv 0$, if $\|u(\cdot,t)\|_\infty<\epsilon_0$
for all $t$ satisfying  $\tilde t_0 \le t\le \tilde t_1$ for some $\tilde t_0\ge t_0+t^5(u_0,v_0)$, then}
$$
\|v(\cdot,t)\|_{W^{2,\infty}}<C^*\epsilon_0\quad \forall \, \tilde t_0+T^*\le t\le \tilde t_1.
$$

First, by \eqref{m2-eq} and \eqref{new-new-eq3-1},
$$
\|u(\cdot,t)\|_\infty\le M_2\quad {\rm and}\quad \|v(\cdot,t)\|_{W^{2,\infty}}\le \tilde C_3\quad \forall\, t\ge \tilde  t_0.
$$
Then, by the definition of $T^*$, \eqref{new-new-eq1}, \eqref{new-new-eq2}, and \eqref{new-new-eq3},
for any $t\ge \tilde t_0+T^*$,
\begin{align}
\label{new-new-eq1-2}
\|\nabla v(t)\|_{L^p(\Omega)}\le  C_1(p)\epsilon_0 +  C_1(p)\sup_{\tilde t_0\le s\le t}\|u(\cdot,s)\|_\infty,
\end{align}
\begin{align}
 \label{new-new-eq2-2}
 &\|A^\eta u(\cdot,t)\|_{L^p(\Omega)}\nonumber\\
 &\le C_2(p,\eta) \epsilon_0+C_2(p,\eta)\sup_{\tilde t_0\le s\le t}\|u(\cdot,s)\|_\infty\big(1+\sup_{\tilde t_0\le s\le t}\|\nabla v(\cdot,s)\|_{L^p(\Omega)}+\sup_{\tilde t_0\le s\le t}\|u(\cdot,s)\|_\infty\big),
 \end{align}
 and
\begin{align}
\label{new-new-eq3-2}
\|v(\cdot,t)\|_{W^{2,\infty}(\Omega)}&\le  C_3(p,\theta,\eta)\epsilon_0 +C_3(p,\theta,\eta)\sup_{\tilde t_0\le s\le t} \|A^\eta u(\cdot,s)\|_{L^p(\Omega)}.
\end{align}
By \eqref{new-new-eq1-2}-\eqref{new-new-eq3-2}, we have
$$
\|v(\cdot,t)\|_{W^{2,\infty}}<C^*\epsilon_0\quad \forall \, t\in [\tilde t_0+T^*,\tilde t_1].
$$

\medskip

\noindent {\bf Step 2.} {\it   In this step, we prove that   for any $t_0\in\RR$ and  $(u_0,v_0) \in C(\bar{\Omega})\times W^{1,\infty}(\Omega)$ with $u_0,v_0 \geq 0$ and $u_0\not \equiv 0$, there is $t_n\to\infty$ such that
$\|u(\cdot,t_n;t_0,u_0,v_0)\|_\infty>\epsilon_0$.}

\medskip

We prove it by contradiction.
Assume that there is $\tilde t_0\ge t_0+t^5(u_0,v_0)$ such that $\|u(\cdot,t)\|_\infty\le \epsilon_0$ for all $t\ge \tilde t_0$.
By Step 1,
\begin{equation}\label{persisetence-full-lemma-05-eq00}
\|v(\cdot,t)\|_{W^{2,\infty}(\Omega)}\leq C^* \epsilon_0 \quad \forall\,  t>\tilde t_0+T^*.
\end{equation}
Thus  by the first equation of \eqref{u-v-eq00}, we get for $t>\tilde t_0+T^*$ that
\begin{align}
\label{persisetence-full-lemma-05-eq01}
  u_t & =\Delta u-\chi\nabla\cdot (u \nabla v)+u\Big(a_0(t,x)-a_1(t,x)u-a_2(t,x)\int_{\Omega}u\Big) \nonumber\\
   & =\Delta u-\chi\nabla u\cdot \nabla v-\chi u\Delta v+u\Big(a_0(t,x)-a_1(t,x)u-a_2(t,x)\int_{\Omega}u\Big) \nonumber\\
   & =\Delta u-\chi\nabla u\cdot \nabla v+u\Big(a_0(t,x)-\chi\Delta v-a_1(t,x)u-a_2(t,x)\int_{\Omega}u\Big)\nonumber \\
   & \geq \Delta u-\chi\nabla u\cdot \nabla v+u\Big(a_{0,\inf}-\chi\Delta v-a_{1,\sup}u-(a_{2,\sup})_+\int_{\Omega}u\Big)\nonumber \\
    & \geq \Delta u-\chi\nabla u\cdot \nabla v+u\Big(a_{0,\inf}-\chi C^*\epsilon_0 -a_{1,\sup}\epsilon_0-|\Omega|(a_{2,\sup})_+\epsilon_0\Big).
\end{align}
Note that
\begin{equation}\label{persisetence-full-lemma-05-eq02}
\delta^*:=\inf u(\cdot,\tilde t_0+T^*)>0.
\end{equation}
Therefore by the comparison principle for parabolic equations, we get
$$
u(x,t)\geq \delta^* e^{(a_{0,\inf}-(a_{1,\sup}+\chi C^*+|\Omega|(a_{2,\sup})_+)\epsilon)(t-\tilde t_0-T^*)}, \quad \forall t>\tilde t_0+T^*.
$$
Note that  $\epsilon_0<\frac{a_{0,\inf}}{a_{1,\sup}+\chi C^*+|\Omega|(a_{2,\sup})_+}$.  We get as $t \to \infty$ in the above equation  that $\lim_{t \to \infty}\|u(\cdot,t)\|_\infty=\infty$, which is a contradiction. Hence the statement in step 2 is true.

\medskip

\noindent {\bf Step 3.}  {\it   In this step, we prove that  there
 is $\delta_{\epsilon_0}>0$ such that for any given $t_0\in\RR$ and $(u_0,v_0) \in C(\bar{\Omega})\times W^{1,\infty}(\Omega)$ with $u_0,v_0 \geq 0$ , for any $t>t_0+{t}^5(u_0,v_0)+1$, if $\sup_{x\in\Omega}u(x,t)=\epsilon_0$,
then $\inf_{x\in\Omega, s\in [t,t+T^*]} u(x,s)\ge \delta_{\epsilon_0}.$
}

\medskip

 We prove it by contradiction. Suppose by contradiction that the statement in Step 2 does not hold. Then there exist   $u_n \in C^0(\bar{\Omega}),$ $v_n \in W^{1,\infty}(\Omega),$ $t_{0n}, t_n, \tilde t_n \in \mathbb{R}$ with $t_n>t_{0n}+{t}^5(u_n,v_n)+1$,
  $\tilde t_n\in [t_n,t_n+T^*]$, $x_n, x^*_n  \in \Omega$ such that
 \begin{equation}\label{persisetence-full-lemma-04-eq-00}
\lim_{n \to \infty}u(x_n,t_n;t_{0n},u_n,v_n)=\epsilon_0,
\end{equation}
 and
 \begin{equation}\label{persisetence-full-lemma-04-eq-01}
\lim_{n \to \infty}u( x^*_n,\tilde t_n;t_{0n},u_n,v_n)=0.
\end{equation}
Since  $t_n>t_{0n}+t^5(u_n,v_n)+1$, by Lemmas \ref{persisetence-full-lemma-00}- \ref{persisetence-full-lemma-03} and Corollaries \ref{cor1}-\ref{cor3}, without loss of
generality, we may assume that
$$
u(\cdot,t_n-1;t_{0n},u_n,v_n)\to u_0^*\quad \text{in}\quad C^0(\bar{\Omega}) \quad \text{and}\quad v(\cdot,t_n-1;t_{0n},u_n,v_n)\to v^*_0\quad\text{in}\quad W^{1,\infty}(\bar{\Omega})
$$
and
$$
u(\cdot,t_n;t_{0n},u_n,v_n)\to u^*\quad \text{in}\quad C^0(\bar{\Omega}) \quad \text{and}\quad v(\cdot,{ t_n};t_{0n},u_n,v_n)\to v^*\quad\text{in}\quad W^{1,\infty}(\bar{\Omega})
$$
as $n\to\infty$. Without loss of generality, we may also assume that
\vspace{-0.05in}$$
a_i(t+t_n-1,\cdot)\to a_i^*(t,x)
\vspace{-0.05in}$$
as $n\to\infty$ locally uniformly in $(t,x)\in\RR\times\bar\Omega$. Then by Lemma \ref{dependence-on-parameters-lm} {{together}  with the generalized Gronwall's inequality (see \cite[Lemma 7.1.1]{DH77}}), we have
\begin{equation}\label{persisetence-full-lemma-04-eq-02}
(u(\cdot,t+t_n-1;t_{0n},u_n,v_n), v(\cdot,t+t_n-1;t_{0n},u_n,v_n))\to  (u^*(\cdot,t;0, u^*_0, v^*_0),v^*(\cdot,t;0, u^*_0, v^*_0))
\end{equation}
as $n\to\infty$ {  locally uniformly in $(t,x)\in\RR\times\bar\Omega$}, where $(u^*(x,t;0,u^*_0,v^*_0), v^*(x,t;0,u^*_0,v^*_0))$  is the solution of \eqref{u-v-eq00} with $a_i(t,x)$ being replaced by $a^*(t,x)$. This implies that
$$
(u^*(x,1;0,u^*_0,v^*_0), v^*(x,1;0,u^*_0,v^*_0))=(u^*(x),v^*(x)).
$$
By \eqref{persisetence-full-lemma-04-eq-00} and the comparison principle for parabolic equations, we have
$$
\inf_{x\in\Omega} u^*(x)>0.
$$
This together with the comparison principle for parabolic equations implies that
$$
\inf_{x\in\Omega,t\in [1,1+T^*]} u^*(x,t;0,u^*_0,v^*_0)>0.
$$
{Let
$$\tilde \epsilon=\frac{\inf_{x\in\Omega,t\in [1,1+T^*]} u^*(x,t;0,u^*_0,v^*_0)}{2}.
 $$
Consider $(x,t) \in \bar\Omega \times [1,1+T^*] $(which is a compact subset of $\bar\Omega\times \RR$). By equation \eqref{persisetence-full-lemma-04-eq-02} there exists $\tilde n=n(\tilde \epsilon)$ such that
\begin{align*}
u(x,t+t_n-1;t_{0n},u_n,v_n)&\geq u^*(x,t;0, u^*_0, v^*_0)-\tilde\epsilon\\
&\geq \inf_{x\in\Omega,t\in [1,1+T^*]} u^*(x,t;0,u^*_0,v^*_0)-\tilde\epsilon\\
&=\frac{\inf_{x\in\Omega,t\in [1,1+T^*]} u^*(x,t;0,u^*_0,v^*_0)}{2}\,\,\, \forall (x,t) \in \bar\Omega \times [1,1+T^*]\, \,\text{and}\, \,\forall\, n>\tilde n.
\end{align*}
Therefore
$$
\inf_{x\in\Omega,t\in [1,1+T^*]} u(x,t+t_n-1;t_{0n},u_n,v_n)\ge \frac{1}{2}\inf_{x\in\Omega,t\in [1,1+T^*]} u^*(x,t;0,u^*_0,v^*_0)\quad \forall n>\tilde n.
$$
Note that
$$\inf_{x\in\Omega,t\in [1,1+T^*]} u(x,t+t_n-1;t_{0n},u_n,v_n)=\inf_{x\in\Omega,t\in [t_n,t_n+T^*]} u(x,t;t_{0n},u_n,v_n).
 $$
 This implies that
 $$
\liminf_{n\to\infty} \inf_{x\in\Omega,t\in [t_n,t_n+T^*]} u(x,t;t_{0n},u_n,v_n)\ge \tilde\epsilon,
 $$
 which contradicts
 \eqref{persisetence-full-lemma-04-eq-01}.
 Thus the above claim follows.
}
\medskip

\noindent {\bf Step 4.}  {\it Let $T^{**}>T^*$ be such that
$$
\delta_{\epsilon_0} e^{(a_{0,\inf}-(a_{1,\sup}+\chi C^*+|\Omega|(a_{2,\sup})_+)\epsilon^*)(T^{**}-T^*)}\ge \epsilon_0.
$$
 In this step, we prove that  for any $t_0\in\RR$ and  $(u_0,v_0) \in C(\bar{\Omega})\times W^{1,\infty}(\Omega)$ with $u_0,v_0 \geq 0$ and $u_0\not \equiv 0$, if  $t_2>t_1>t_0+t^5(u_0,v_0)$ are such that
$\|u(\cdot,t_i;t_0,u_0,v_0)\|=\epsilon_0$ $(i=1,2)$ and $\|u(\cdot,t;t_0,u_0,v_0)\|_\infty<\epsilon_0$ for $t\in (t_1,t_2)$, then
$t_2-t_1\le T^{*}+T^{**}$.}
\medskip

In fact, by the arguments in Steps 1-3, we have
$$
u(x,t)\geq \delta_{\epsilon_0} e^{(a_{0,\inf}-(a_{1,\sup}+\chi C^*+|\Omega|(a_{2,\sup})_+)\epsilon)(t-t_1-T^*)}, \quad \forall t\in (t_1+T^*,t_2).
$$
It then follows that
$$
t_2\le t_1+T^*+T^{**}.
$$

\medskip

\noindent{\bf Step 5.} {\it In this step, we prove that there is $\eta>0$ such that for any $t_0\in\RR,$ $u_0\in C^0(\bar \Omega),$ $v_0 \in W^{1,\infty}(\Omega)$ with $u_0\ge 0,$ and $u_0\not \equiv 0$, there is $\tau(u_0,v_0)>0$ such that
\begin{equation}
\label{new-new-eq4}
 u(x,t;t_0,u_0,v_0) \ge  \eta,\quad \forall\,\, t\ge t_0+\tau(u_0,v_0).
\end{equation}
}

 First, { by the arguments in  Step 3}, there is $\eta>0$ such that  for any given $t_0\in\RR$ and $(u_0,v_0) \in C(\bar{\Omega})\times W^{1,\infty}(\Omega)$ with $u_0,v_0 \geq 0$ , for any $t>t_0+{t}^5(u_0,v_0)+1$, if $\sup_{x\in\Omega}u(x,t)=\epsilon_0$,
then
$$\inf_{x\in\Omega, s\in [t,t+T^*+T^{**}]} u(x,s)\ge \eta.
$$

 Next, by Step 4, if $t>t_0+t^5(u_0,v_0)+T^*+T^{**}$ is such that
 $\|u(\cdot,t;t_0,u_0,v_0)\|<\epsilon_0$, then there are $t_1,t_2$ with
   $t_2>t>t_1>t_0+t^5(u_0,v_0)+1$  such that
$\|u(\cdot,t;t_0,u_0,v_0)\|_\infty<\epsilon_0$ for $t\in (t_1,t_2)$,  $\|u(\cdot,t_i;t_0,u_0,v_0)\|_\infty=\epsilon_0$ ($i=1,2$), and
$t_2-t_1\le T^{**}+T^*$. Hence
$$
\inf_{x\in\Omega,t\in [t_1,t_2]} u(x,t)\ge \eta.
$$

It then follows that the statement in Step 5 holds and the theorem is proved.
\end{proof}

\section{Strictly positive entire solutions}

In this section, we investigate the existence of strictly positive entire solutions of  \eqref{u-v-eq00} and prove Theorem \ref{thm-entire-solution-00}.

 \begin{proof}[Proof of Theorem \ref{thm-entire-solution-00}]
   First of all, fix $u_0 \in C^0(\Omega)$ , $v_0 \in W^{1,\infty}(\Omega)$ with $u_0,v_0 \geq 0$ and $\inf u_0>0.$
   By Theorem \ref{thm-persistence-entire solution-000},  there is $\tau(u_0,v_0)$ such that
   \begin{equation}
   \label{entire-solution-eq01}
   u(x,t;t_0,u_0,v_0)\ge \eta\quad \forall \, x\in\bar\Omega,\,\, t\ge t_0+\tau(u_0,v_0),\,\, t_0\in\RR.
   \end{equation}

Next, for each $n>\tau(u_0,v_0),$
 let
 $$u_n(\cdot)=u(\cdot,0;-n,u_0,v_0)\quad {\rm and}\quad
 v_n(\cdot)=v(\cdot,0;-n,u_0,v_0).
 $$
  Without loss of generality, we may assume that
  there exist $u^*_0 \in C^0(\Omega)$ and $v^*_0 \in W^{1,\infty}(\Omega)$ such that
  $$(u_n,v_n) \to (u^*_0, v^*_0)\quad {\rm in}\quad C^0(\Omega)\times W^{1,\infty}(\Omega)$$
  as $n \to \infty.$
  We then have $u(\cdot,t;-n,u_0,v_0)=u(\cdot,t;0,u(\cdot,0;-n,u_0,v_0),v(\cdot,0;-n,u_0,v_0))$ and for $t>0,$ $u(\cdot,t;-n,u_0,v_0)\to u^*(\cdot,t;0,u^*_0,v^*_0)$ as $n \to \infty.$

   We now claim that $u^*(\cdot,t;0,u^*_0,v^*_0)$ has backward extension. Indeed fix $m \in \mathbb{N},$ and for each $n>m+\tau(u_0,v_0),$ let $u^m_n(\cdot)=u(\cdot,-m;-n,u_0,v_0)$ and $v^m_n(\cdot)=v(\cdot,-m;-n,u_0,v_0).$ Then there exist $n_k\to \infty$,  $u^*_{m,0} \in C^0(\Omega)$ and $v^*_{m,0} \in W^{1,\infty}(\Omega)$ such that $u^m_{n_k} \to u^*_{m,0}$ and $v^m_{n_k} \to v^*_{m,0}$ as $n \to \infty.$ We have for $t>-m,$
   \begin{align*}
   u(\cdot,t;-n_k,u_0,v_0)&=u(\cdot,t;-m,u(\cdot,-m;-n_k,u_0,v_0),v(\cdot,-m;-n_k,u_0,v_0))\\
   &\to u^*_m(\cdot,t;-m,u^*_{m,0},v^*_{m,0})
    \end{align*}
    as $n_k \to \infty.$
    And for $t>0,$ $u^*(\cdot,t;0,u^*_0,v^*_0)=u^*_m(\cdot,t;-m,u^*_{m,0},v^*_{m,0}).$ Thus $u^*(\cdot,t;0,u^*_0,v^*_0)$ has backward extension.

   Finally,  fix $t\in \mathbb{R} $ and choose $m \in \mathbb{N}$ such that $t>-m+\tau(u_0,v_0).$ Then by equation \eqref{entire-solution-eq01}, we get
$$\eta  \leq u^*(\cdot,t;0,u^*_0,v^*_0)=u^*_m(\cdot,t;-m,u^*_{m,0},v^*_{m,0}),\,\, \forall x \in \bar{\Omega}.$$
This completes the proof.
 \end{proof}

\end{document}